\pdfoutput=1
\documentclass[11pt,a4paper,reqno]{amsart}
\usepackage{color}
\usepackage{amsfonts,amsmath,amssymb,amsxtra,url,float} 
\allowdisplaybreaks[4] 
\usepackage[colorlinks,linkcolor=RoyalBlue,anchorcolor=Periwinkle,citecolor=Orange,urlcolor=Green]{hyperref} 
\usepackage[usenames,dvipsnames]{xcolor} 
\usepackage{enumitem}
\usepackage{geometry} 
\usepackage{rotating} 
\usepackage{lscape} 
\usepackage{multirow}
\usepackage{graphicx} 
\usepackage{subfigure} 
\usepackage{tikz-cd}
\usepackage{tikz}
\usetikzlibrary{decorations.pathreplacing,decorations.markings}
 \tikzset{
  on each segment/.style={
    decorate,
    decoration={
      show path construction,
      moveto code={},
      lineto code={
        \path [#1]
        (\tikzinputsegmentfirst) -- (\tikzinputsegmentlast);
      },
      curveto code={
        \path [#1] (\tikzinputsegmentfirst)
        .. controls
        (\tikzinputsegmentsupporta) and (\tikzinputsegmentsupportb)
        ..
        (\tikzinputsegmentlast);
      },
      closepath code={
        \path [#1]
        (\tikzinputsegmentfirst) -- (\tikzinputsegmentlast);
      },
    },
  },
  mid arrow/.style={postaction={decorate,decoration={
        markings,
        mark=at position 0.6 with {\arrow[#1]{stealth}} 
      }}},
}
\usetikzlibrary{arrows}
\usetikzlibrary{trees}

\usetikzlibrary{matrix}
\usetikzlibrary{patterns}
\usetikzlibrary{shadings} 
\usepackage[all]{xy} 
\usepackage{mathdots} 
\usepackage{dsfont} 
\usepackage{cite}
\usepackage{mathrsfs} 
\usepackage{multicol} 
\numberwithin{figure}{section}
\usepackage{marginnote} 
\usepackage{graphicx} 
\usepackage{multicol} 

\newcommand{\checks}[1]{{\color{black}{#1}}} 

\newtheorem{theorem}{Theorem}[section]
\newtheorem{lemma}[theorem]{Lemma}
\newtheorem{corollary}[theorem]{Corollary}
\newtheorem{main theorem}[theorem]{Main Theorem}
\newtheorem{proposition}[theorem]{Proposition}
\newtheorem{property}[theorem]{Proposition}
\newtheorem{definition}[theorem]{Definition}

\newtheorem{remark}[theorem]{Remark}
\newtheorem{example}[theorem]{Example}
\newtheorem{notation}[theorem]{Notation}

\newtheorem{prop-coro}[theorem]{Corollary/Proposition}

\usetikzlibrary{arrows}

\numberwithin{equation}{section}

\def\<{\langle} 
\def\>{\rangle} 
\def\NN{\mathbb{N}} 

\def\I{\mathop{\rm \mathcal{I}}\nolimits}

\def\extdim{\mathop{\rm ext.dim}\nolimits}

\def\id{\mathop{\rm id}\nolimits}
\def\pd{\mathop{\rm pd}\nolimits}

\def\mod{\mathsf{mod}}

\def\add{\mathsf{add}}

\def\gldim{\mathop{\rm gl.dim}\nolimits}

\def\rep{\mathop{{\rm rep}}\nolimits}

\def\End{\mathop{\rm End}\nolimits}
\def\rep{\mathrm{rep}}
\def\add{\mathop{\rm add}\nolimits}

\def\Hom{\mathop{\rm Hom}\nolimits}
\def\End{\mathop{\rm End}\nolimits}

\def\Ker{\mathop{\rm Ker}\nolimits}
\def\Im{\mathop{\rm Im}\nolimits}
\def\rad{\mathop{\rm rad}\nolimits}
\def\soc{\mathop{\rm soc}\nolimits}
\def\top{\mathop{\rm top}\nolimits}

\def\gldim{\mathop{\rm gl.dim}\nolimits}

\def\extdim{\mathop{\rm ext.dim}\nolimits}
\def\itLamb{\mathit{\Lambda}}

\usepackage{marvosym} \newcommand{\gent}{\text{\Clocklogo}}
\newcommand{\ospec}{{\rm OSpec}}
\newcommand{\udim}{{\rm u.dim}}

\def\LL{\mathop{\rm \ell\ell}\nolimits}

\newcommand{\edge}{\ar@{-}}

\def\Pic{Figure}
\def\kk{\mathbf{k}} \def\Q{\mathcal{Q}} \def\I{\mathcal{I}}
\def\A{\mathbb{A}}  
\def\dA{\overrightarrow{\mathbb{A}}}  
\def\s{\mathfrak{s}} \def\t{\mathfrak{t}}
\def\Sub{\mathsf{Sub}}
\def\Fac{\mathsf{Fac}}
\def\Aus{\mathsf{Aus}}
\def\Gh{\mathsf{Gh}}
\def\Obj{\mathrm{Obj}}
\def\circ{\ \lower-0.2ex\hbox{\tikz\draw (0pt, 0pt) circle (.1em);} \ }
\def\circled{\ \lower-0.2ex\hbox{\tikz\draw (0pt, 0pt) circle (.1em);} \ }
\def\multi{\ \lower-0.2ex\hbox{\tikz\fill (0pt, 0pt) circle (.1em);} \ }

\def\F{\mathcal{F}}
\def\T{\mathcal{T}}
\def\V{\mathcal{V}}
\def\simp{\mathsf{S}}

\newcommand{\To}[1]{\mathop{\longrightarrow}\limits^{#1}}
\newcommand{\defines}{\it\color{black}}

\begin{document}

\title{The Orlov spectra of Abelian categories$^{*}$}
\author{Junling Zheng$^{1}$}

\author{Yu-Zhe Liu$^{2,\dag}$}

\thanks{$*$
}

\thanks{$1$ Department of Mathematics, China Jiliang University, Hangzhou, 310018, Zhejiang Province, P. R. China.
        E-mail: zhengjunling@cjlu.edu.cn / zjlshuxue@163.com. }
\thanks{$2$ School of Mathematics and statistics, Guizhou University, Guiyang 550025, Guizhou Province, P. R. China.
        E-mail: liuyz@gzu.edu.cn / yzliu3@163.com }
\thanks{$\dag$ Corresponding author.}






\maketitle

\begin{abstract}
We introduced the notion of Orlov spectra of Abelian categories, and study its some properties.
In particular, we give precise result of Orlov spectra of
algebras with type $\mathbb{A}_{n}$.
\end{abstract}

\setcounter{tocdepth}{3}
\setcounter{secnumdepth}{3}
\tableofcontents

\pagestyle{myheadings}
\markboth{\rightline {\scriptsize  J. L. Zheng, Y. Z Liu\emph{}}}
         {\leftline{\scriptsize  The Orlov spectra of Abelian categories }}


\section{Introduction}

Rouquier introduced in \cite{rouquier2006representation,rouquier2008dimensions} the dimension of a triangulated category under the idea of Bondal and van den Bergh in \cite{bondal2003generators}.
Roughly speaking, it is an invariant that measures how quickly the category can be built from one object. This dimension plays an important role in representation theory
\cite{ballard2012orlov,
bergh2015gorenstein,
chen2008algebras,
han2009derived,
oppermann2012generating,
rouquier2006representation,
rouquier2008dimensions,
zheng2020upper}.
For example, it can be used to compute the representation dimension of Artin algebras \cite{rouquier2006representation, oppermann2009lower}.
Similar to the dimension of triangulated categories, the (extension) dimension of an Abelian category was introduced by Beligiannis in \cite{beligiannis2008some}, also see \cite{dao2014radius}. The size of the extension dimension reflects how far an Artin algebra is from a finite representation type, some relate result can be seen in \cite{zheng2019extension}.

The spectrum of a triangulated category is an invariant of the category introduced by Orlov in 2009 (see \cite{orlov2009}). It can be seen as a generalization of the Rouquier dimension for triangulated categories. In 2012, Ballard, Favero, and Katzarkov systematically studied this invariant and named it the Orlov spectrum (see \cite{ballard2012orlov}). Similar to the Orlov spectrum of triangulated categories, we will define the Orlov spectrum for Abelian categories, which can be regarded as a generalization of the extension dimension for Abelian categories.
In this paper, we will mainly investigate some properties of the Orlov spectrum of finite representation type algebras.
We used the properties of coghost to demonstrate our main results, where coghost is a dual of ghost which is introduced by Lank in \cite{Lank2023}.
The main theorems of our paper are as follows:

\begin{theorem} \label{thm:1.1} \

{\rm(1)} {\rm(see Theorem \ref{thm:main 1})}
Let $0  \longrightarrow L\longrightarrow  M\longrightarrow N\longrightarrow 0$
be an exact sequence in $\mod\itLamb$. Then
\[\max\{\LL^{t_{\simp}}(L), \LL^{t_{\simp}}(N)\}  \leqslant \LL^{t_{\simp}}(M) \leqslant \LL^{t_{\simp}}(L)+\LL^{t_{\simp}}(N),\]
where $\LL^{t_{\simp}}$ is $t_{\simp}$-radical layer length {\rm(}see Definition \ref{radical-length}{\rm)}.
In particular, if $\LL^{t_{\simp}}(L)=0$, then $\LL^{t_{\simp}}(N)=\LL^{t_{\simp}}(M);$
if $\LL^{t_{\simp}}(N)=0$, then $\LL^{t_{\simp}}(L)=\LL^{t_{\simp}}(M)$.

{\rm(2)} {\rm(see Theorem \ref{thm:main 2})}
If $\itLamb$ is representation-finite, then
\[\Big\{
         \Big\lceil
            \dfrac{\LL^{t_{\simp}}(\itLamb)}{d}
         \Big\rceil
         -1
     \ \Big|\
         d \in \NN
         \text{ and }
         1 \leqslant d < \LL^{t_{\simp}}(\itLamb)
 \Big\}
\subseteq \ospec(\itLamb),\]
where $\lceil \alpha\rceil$ is the least integer greater than $\alpha$.
\end{theorem}

Furthermore, we obtain the following result by using Theorem \ref{thm:1.1}, cf. Propositions  \ref{prop-An} and \ref{prop:Tm-cog}.

\begin{theorem}[{Theorem \ref{thm:main 3}}]
Let $\dA_n$ be the linearly oriented quiver with type $\A_n$, i.e.,
\[\dA_n = \xymatrix{1\ar[r]^{a_1} & 2 \ar[r]^{a_2} & \cdots \ar[r]^{a_{n-1}} & n}. \]
Then the Orlov spectrum $\ospec(A_n)$ of $A_n:=\kk\dA_n$ is $\{0,1,2,\cdots,n-1\}.$
\end{theorem}

\section{Preliminaries}

Let $\mathcal{A}$ be an Abelian category. In this paper, we assume that
all subcategories of $\mathcal{A}$ are full, additive and closed under isomorphisms and all functors between categories are additive.
For a subclass $\mathcal{U}$ of $\mathcal{A}$, we use $\add \mathcal{U}$ to
denote the subcategory of $\mathcal{A}$ consisting of
direct summands of finite direct sums of objects in $\mathcal{U}$,
\checks{and use $\langle\mathcal{U}\rangle$ to denote the subcategory of $\mathcal{A}$ containing all objects lying in $\mathcal{U}$}.
For two objects $X$ and $Y$ in $\mathcal{A}$, we denote by $X\le Y$ if $X$ is a subobject of $Y$, and denote by $X\le_{\oplus} Y$ if $X$ is a direct summand of $Y$.
Let $\NN$ be the set of natural numbers and \checks{$\mathbb{N}^{+}=\NN\backslash\{0\}$}.

\subsection{The extension dimension of Abelian category}

Let $\mathcal{U}_1,\mathcal{U}_2,\cdots,\mathcal{U}_n$ be subcategories of $\mathcal{A}$. Define
\begin{align*}
   \mathcal{U}_1\multi \mathcal{U}_2
:= {\add}\{X\in \mathcal{A}\mid & {\rm there \;exists \;an\; sequence \;}
   0\rightarrow U_1\rightarrow  X \rightarrow U_2\rightarrow 0\  \\
& {\rm in}\ \mathcal{A}\ {\rm with}\; U_1 \in \mathcal{U}_1 \;{\rm and}\; U_2 \in \mathcal{U}_2\}.
\end{align*}
For any subcategories $\mathcal{U},\V$ and $\mathcal{W}$ of $\mathcal{A}$, by \cite[Proposition 2.2]{dao2014radius} we have
$$(\mathcal{U}\multi\V)\multi\mathcal{W}=\mathcal{U}\multi(\V\multi\mathcal{W}).$$
Inductively, we can define
\begin{align*}
\mathcal{U}_{1}\multi  \mathcal{U}_{2}\multi \dots \multi\mathcal{U}_{n}:=
\add \{X\in \mathcal{A}\mid {\rm there \;exists \;an\; sequence}\
0\rightarrow U\rightarrow  X \rightarrow V\rightarrow 0  \\{\rm in}\ \mathcal{A}\ {\rm with}\; U \in \mathcal{U}_{1} \;{\rm and}\;
V \in  \mathcal{U}_{2}\multi \dots \multi\mathcal{U}_{n}\}.
\end{align*}
For a subcategory $\mathcal{U}$ of $\mathcal{A}$, set
$[\mathcal{U}]_{0}=0$, $[\mathcal{U}]_{1}=\add\mathcal{U}$,
$[\mathcal{U}]_{n}=[\mathcal{U}]_1\multi [\mathcal{U}]_{n-1}$ for each $n\geqslant 2$,
and $[\mathcal{U}]_{\infty}=\mathop{\bigcup}_{n\geqslant 0}[\mathcal{U}]_{n}$, cf. \cite{beligiannis2008some}.

\begin{example} \rm \label{exp:A4}
Notice that $[\mathcal{U}]_{n}$ and $[\mathcal{U}]_{n-1}\multi [\mathcal{U}]_{n-1}$ are generally not equal. 
To avoid this misunderstanding, we provide this example by using
\begin{center}
  $A_4 = \kk\Q$ with $\Q = \xymatrix{1 \ar[r] & 2 \ar[r] & 3 \ar[r] & 4}$.
\end{center}
Take $S = S(1)\oplus S(2) \oplus S(3) \oplus S(4)$ be the direct sum of simple $A_4$-modules $S(1)$, $S(2)$, $S(3)$, and $S(4)$, and define $[\mathcal{U}]_1 = \add S$.
Then $[\mathcal{U}]_2 = [\mathcal{U}]_1\multi [\mathcal{U}]_1$ consist of all indecomposable $A_4$-modules with dimensions $\leqslant 3$, that is,
\[ [\mathcal{U}]_2 = \{ S(1), S(2), S(3), S(4),
\big({_4^3}\big), \big({_3^2}\big), \big({_2^1}\big) \}. \]
For $[\mathcal{U}]_3$, a customary misunderstanding is that $[\mathcal{U}]_3 = [\mathcal{U}]_2\multi [\mathcal{U}]_2$.
Indeed, by the definition of $[\mathcal{U}]_{n}=[\mathcal{U}]_1\multi [\mathcal{U}]_{n-1}$ and $P(1) \in [\mathcal{U}]_2\multi [\mathcal{U}]_2$, we have
\begin{align*}
    [\mathcal{U}]_3
& = [\mathcal{U}]_1 \multi [\mathcal{U}]_2 = [\mathcal{U}]_2 \multi [\mathcal{U}]_1 \\
& = \Big\{ S(1), S(2), S(3), S(4),
       \big({_4^3}\big), \big({_3^2}\big), \big({_2^1}\big),
       \Big(\begin{smallmatrix} 1\\ 2\\ 3 \end{smallmatrix}\Big),
       \Big(\begin{smallmatrix} 2\\ 3\\ 4 \end{smallmatrix}\Big) \Big\} \\
& (\ne [\mathcal{U}]_2\multi [\mathcal{U}]_2)
\end{align*}
whose elements are the objects to be obtained by the extensions of
a module lying in $[\mathcal{U}]_1$ and a module lying in $[\mathcal{U}]_2$. 
We have $P(1) \notin [\mathcal{U}]_3$. Similarly, one can check that
\[ [\mathcal{U}]_4 = [\mathcal{U}]_1\multi [\mathcal{U}]_3 = [\mathcal{U}]_2 \multi [\mathcal{U}]_2 = [\mathcal{U}]_3 \multi [\mathcal{U}]_1 = \mod A_4 \]
because $P(1)  \in [\mathcal{U}]_4$.
\end{example}

The following two lemmas are useful for this paper.

\begin{lemma}[\!{\cite[Lemma 2.3]{zheng2019extension}}] \label{lem2.4a}
Let $\mathcal{A}$ be an Abelian category and
$T_{1},T_{2}\in \mathcal{A}$ and $m,n\in\mathbb{N}$.
Then

$(1)$ $[T_{1}]_{m}\multi [T_{2}]_{n}\subseteq [T_{1}\oplus T_{2}]_{m+n};$

$(2)$
$[T_{1}]_{m}\oplus  [T_{2}]_{n}\subseteq [T_{1}\oplus T_{2}]_{\max\{m,n\}}.$
\end{lemma}

\begin{lemma}[\!{\cite[Lemma 2.4]{zheng2019extension}}] \label{lem2.4}
 Let $F:\mathcal{A}\to \mathcal{B}$ be an exact functor of Abelian categories.
 Then $F([T]_{n})\subseteq [F(T)]_{n}$ for any $T\in\mathcal{A}$ and positive integer $n\geqslant 1.$
\end{lemma}

\begin{definition}\label{def2.1} \rm
Let $X$ be an object in $\mathcal{A}$. The generation time of $X$,
denoted $\gent_{\mathcal{A}}(X)$ is
\[\gent_{\mathcal{A}}(X) :=
\begin{cases}
  \inf\{n\in\NN \mid \mathcal{A}=[X]_{n+1}\},
& \text{if } [X]_t=\mathcal{A} \text{ for some } t\in\NN; \\
  +\infty,
& \text{otherwise}.
\end{cases} \]
The object $X$ is called a {\defines strong generator} of $\mathcal{A}$ if $\gent_{\mathcal{A}}(X)$ is finite.
\end{definition}

\begin{example} \label{exp:A4-gener time} \rm
Consider the algebra $A_4$ given in Example \ref{exp:A4} and keep the notations from this example.
Then we have $\gent_{\mod A_4}(S) = 4-1 = 3 < +\infty$. Thus, $S$ is a strong generator of $\mod A_4$.
\end{example}

\begin{lemma}\label{time1}
For an object $X$ in $\mathcal{A}$ with $\gent_{\mathcal{A}}(X)<+\infty$, we have
\[\gent_{\mathcal{A}}(X)\geqslant \gent_{\mathcal{A}}(X\oplus Y)\]
holds for arbitrary object $Y$ in $\mathcal{A}$.
\end{lemma}

\begin{proof}
By $\add X = [X]_{1}\subseteq [X\oplus Y]_{1} = \add (X\oplus Y)$ and induction, we have
\begin{align}\label{formula:time1}
  [X]_{n+1}\subseteq [X\oplus Y]_{n+1}
\end{align}
since, for any integer $m\geqslant 0$, $[X]_{m}\subseteq [X\oplus Y]_{m}$ yields
\begin{align*}
  [X]_{m+1} = [X]_{1}\multi [X]_{m} \subseteq [X\oplus Y]_{1} \multi [X\oplus Y]_{m} = [X\oplus Y]_{m+1}.
\end{align*}
Furthermore, (\ref{formula:time1}) follows
\[ \gent_{\mathcal{A}}(X) = \inf \{n\in\NN \mid [X]_{n+1}=\mathcal{A} \}
\geqslant \inf \{ n\in\NN \mid [X\oplus Y]_{n+1}=\mathcal{A} \} = \gent_{\mathcal{A}}(X\oplus Y) \]
by Definition \ref{def2.1} in the case for $\gent_{\mathcal{A}}(X)<+\infty$.
\end{proof}

\begin{definition} \rm \label{def:OSpec}
The {\defines Orlov spectrum} $\ospec(\mathcal{A})$ of an Abelian category $\mathcal{A}$ is defined as
\[\ospec(\mathcal{A})=\{ \gent_{\mathcal{A}}(X)\;|\;X\text{ is a strong generator of $\mathcal{A}$}\}. \]
The {\defines extension dimension} (\!\!\cite{beligiannis2008some}) of $\mathcal{A}$ is defined as the infimum of $\ospec(\mathcal{A})$,
that is,
$$\extdim \mathcal{A}:=\inf\ospec(\mathcal{A}).$$
Dually, we can define the {\defines ultimate dimension}  of $\mathcal{A}$ is defined as the supremum of $\ospec(\mathcal{A})$,
that is,
$$\udim \mathcal{A}:=\sup\ospec(\mathcal{A}).$$
For Artin algebra $\itLamb$, we set
$\ospec(\itLamb):=\ospec(\mod\itLamb)$, and $\udim(\itLamb):=\udim(\mod\itLamb)$.
\end{definition}

\begin{lemma}\label{belongtoOrlov}
If $[T]_{n+1}=\mathcal{A}\neq[T]_{n}$, then $n\in \ospec(\mathcal{A})$.
\end{lemma}

\begin{proof}
Since $[T]_{1}\subseteq [T]_{2}\subseteq \cdots\subseteq [T]_{n}\neq \mathcal{A}$
and $[T]_{n+1}= \mathcal{A}$, we have $\gent_{\mathcal{A}}(T)=n$.
It follows $n\in \ospec(\mathcal{A})$ by the definition of Orlov spectrum (see Definition \ref{def:OSpec}).
\end{proof}

\begin{example} \label{exp:A4-ospec} \rm
Consider the algebra $A_4$ given in Example \ref{exp:A4} and keep the notations form this example.
By Example \ref{exp:A4-gener time}, we have $3 \in \ospec(A_4)$.
Moreover, by Lemma \ref{belongtoOrlov}, we can also prove $3 \in \ospec(A_4)$
since $[S]_4 = \mod A_4 \ne [S]_3$ holds.
\end{example}

\subsection{Ghosts and coghosts}
This section contains two parts as follows: in the first part we recall contravariantly finite subcategories and covariantly finite subcategories of Abelian categories;
and in the second part we recall ghosts introduced by Beligiannis (see \cite{beligiannis2008some}),
\checks{and coghosts introduced by Lank (see \cite[Definition 4.5]{Lank2023})}.
\checks{The properties of coghosts are similar to that of ghosts.
However, for the convenience of reading, we still provide detailed proofs for properties of coghosts in our paper.}

\subsubsection{Contravariantly finite subcategories and covariantly finite subcategories}
Let $\mathcal{A}$ be an Abelian category and $\mathcal{C}$ be a subcategory of $\mathcal{A}$ in this section.
A map $f:C\to X$ in $\mathcal{A}$ with $C\in\mathcal{C}$
is called {\defines a right $\mathcal{C}$-approximation} of $X$ if for every
$f':C'\to X$ with $C'\in\mathcal{C}$ there exists a map $g: C'\to C$
such that the following diagram
\[\xymatrix{
& C'\ar@{-->}[ld]_{g} \ar[d]^{f'}  \\
  C\ar[r]_{f} & X   \\
 }\]
commutes, that is, $f'=f\circ g$.
If for every $X\in \mathcal{A}$ admits a right $\mathcal{C}$-approximation, then $\mathcal{C}$ is called {\defines contravariantly finite}.
Dually, a map $f: X \to C$ in $\mathcal{A}$ with $C\in\mathcal{C}$
is called a {\defines left $\mathcal{C}$-approximation} of $X$ if for every
$f': X \to C' $ with $C'\in\mathcal{C}$ there exists a map $g: C\to C'$
such that the following diagram
\[\xymatrix{
X\ar[r]^{f}  \ar[d]_{f'}& C\ar@{-->}[ld]^{g}&  \\
C'& \\
 }\]
commutes, that is, $f'=f\circ g$. If for every $X\in \mathcal{A}$ admits a left $\mathcal{C}$-approximation, then $\mathcal{C}$ is called {\defines covariantly finite}.

\subsubsection{Ghost lemma and Coghost lemma}
Let $T$ be an object in an Abelian category $\mathcal{A}$.

\begin{definition}[\!\!{\cite[Definition 1.2]{beligiannis2008some}, \cite[Definition 4.5]{Lank2023}}] \rm
A morphism $f:Y\to A$ (resp. $f:A\to Y$) in $\mathcal{A}$ is called a {\defines $T$-ghost} (resp. {\defines $T$-coghost}) if the induced map $\Hom_{\mathcal{A}}(T',f)$ (resp. $\Hom_{\mathcal{A}}(f,T')$) is zero for each $T'\in\add T$.
Furthermore, if a map can be written a composition of $n$ $T$-coghost maps, then it is called an {\defines $n$-fold ghost} (resp. {\defines $n$-fold coghost}).
\end{definition}

\begin{notation} \rm
For a set $S$ of some objects lying in an Abelian category $\mathcal{A}$, we use $\langle S\rangle$ be the minimal full subcategory containing $S$.
For any $n\in\NN^+$, we set
\begin{itemize}
  \item[(1)] $\Gh^{n}_{T}(X,Y) := \{ f\in\Hom_{\mathcal{A}}(X, Y) \mid
      f\text{ is an } n\text{-fold } T\text{-coghost map} \}$,
  \item[(2)] $\Gh^n_T(-,Y) := \big\langle \Gh^{n}_{T}(X,Y) \mid X\in\Obj(\mathcal{A}) \big\rangle$,
  \item[(3)] and $\Gh^n_T(X,-) := \big\langle \Gh^{n}_{T}(X,Y) \mid Y\in\Obj(\mathcal{A}) \big\rangle$.
\end{itemize}
\end{notation}

\begin{remark} \label{rmk:ideal} \rm
It is easy to see that $\Gh^{n}_{T}(X,Y)$ is an ideals of $\mathcal{A}$, that is,
for any $T$-coghost $f:X\to Y$ and arbitrary homomorphism $g:Y\to Z$ (resp., $g:Z\to X$),
the composition $g \circled f$ (resp., $f\circled g$) is also a $T$-coghost.
\end{remark}

Recall that the {\defines kernel} of a morphism $f:X\to Y$ in an Abelian category $\mathcal{A}$ is the morphism $\eta: K \to X$ from an object $K$ in $\mathcal{A}$ to $X$ such that
\begin{itemize}
  \item[(1)] $f\circled\eta = 0$;
  \item[(2)] and, for any morphism $\eta':K'\to X$ satisfying $f\circled \eta'$, there is a unique morphism $\tau: K'\to K$ such that $\eta'=\eta\circled\tau$.
\end{itemize}
One can check that $K$ is unique up to isomorphism. In most cases, $K$ is written as $\Ker f$.

\begin{lemma} \label{coghost}
The kernel of any left $\mathcal{C}$-approximation in $\mathcal{A}$ is a $\mathcal{C}$-coghost.
\end{lemma}

\begin{proof}
Assume that $f:X\to T$ is a left $\mathcal{C}$-approximation and $\eta:\Ker f \to X$ its kernel.
Here, $T$ is an object lying in $\mathcal{C}$.
For each $h\in\Hom_{\mathcal{A}}(X,T')$, there exists a map $f':T\to T'$ such that
$T'\in\mathcal{C}$ and $h=f'\circled f$ since $f$ is a left $\mathcal{C}$-approximation.
Thus, we obtain the following diagram
\[\xymatrix{
        0\ar[r]& {\rm Ker} f\ar[r]^-{\eta}\ar[d]_{h\circled \eta}&  X\ar[r]^-{f}  \ar[ld]_-{h}& T\ar@{-->}[lld]^-{f'}   \\
               &T'                      &&  \\
 }\]
commutes.
Then
\[h\circled \eta=(f'\circled f)\circled \eta=f'\circled (f\circled \eta)=f'\circled (f\circled \eta)=f'\circled 0 =0.\]
It follows that $\eta:\Ker f \to X $ is a $\mathcal{C}$-coghost by the definition of coghost.
\end{proof}

Let $\Fac(T)$ (resp. $\Sub(T)$) be the subcategory of $\mathcal{A}$ whose object is isomorphic to a quotient (resp. subobject) of some direct sum $T^{\oplus I}$, where $I$ is a finite index set.
The following lemma, i.e., Ghost Lemma, is established by Beligiannis in \cite{beligiannis2008some}.

\begin{lemma}[{Ghost Lemma \cite[Lemma 1.3]{beligiannis2008some}}] \label{ghostlemm} \
\begin{itemize}
  \item[\rm(a)]
    If $X\in[\Fac(T)]_{n}$, then
    $\Gh^{n}_{T}(X,-)=0$.
  \item[\rm(b)]
    If $\add T$ is contravariantly finite in $\mathcal{A}$, then the following
    are equivalent$:$
    \begin{itemize}
      \item[\rm(1)] $\Gh^{n}_{T}(X,-)=0;$
      \item[\rm(2)] $X\in[\Fac(T)]_{n}.$
    \end{itemize}
    \end{itemize}
\end{lemma}

Next, we provide a dual result, say Coghost Lemma, for Ghost Lemma \checks{which is first shown by Lank in \cite[Lemma 4.6]{Lank2023}.}

\begin{lemma}[Coghost Lemma] \label{coghostlemm} \
\begin{itemize}
  \item[\rm(a)]
    If $Y\in[\Sub(T)]_{n}$, then
    $\Gh^{n}_{T}(-,Y)=0$.
  \item[\rm(b)]
    If $\add T$ is covariantly finite in $\mathcal{A}$, then the following
    are equivalent$:$
    \begin{itemize}
      \item[\rm(1)] $\Gh^{n}_{T}(-,Y)=0;$
      \item[\rm(2)] $Y\in[\Sub(T)]_{n}.$
    \end{itemize}
\end{itemize}
\end{lemma}

\checks{For the convenience of readers, we still provide a proof for Coghost Lemma.}
\begin{proof}
(a) If $n=1$. Let $f: X \to Y$ be a $T$-coghost.
Since $Y\in[\Sub(T)]_{1}$, there is a monomorphism $g:Y \to T'$ such that $T'\in \add T$.
And, by the definition of coghost and $T'\in \add T$,
we have $g\circled f=0$ since $f:X\to Y$ is a $T$-coghost.
Thus, $f=0$ by using $g$ to be a monomorphism.

Next, assume that, for any $n\leqslant k$, $Y\in [\Sub(T)]_{k}$ yields that $\Gh^{k}_{T}(-,Y)=0$.
Considering the case of $n=k+1$. we have, for any $Y\in[\Sub(T)]_{k+1}$,
the following short exact sequence
\[ 0 \To{}{} Y_1 \To{g_1}{} Y \To{g_2}{} Y_2 \To{}{} 0 \]
with $Y_{1}\in [\Sub(T)]_{1}$ and $Y_{2}\in [\Sub(T)]_{k}$,
and there is a monomorphism $h_1: Y_1\to T_1$ such that $T_{1}\in\add T$.
Let $f\in \Gh_{T}^{k+1}(-,Y)$.
Then we have $f=f_{1}\circled f_{2}\circled  \cdots f_{k}\circled f_{k+1}$,
where each $f_{i}:X_{i}\to X_{i-1}$ is a $T$-coghost (we set $X_{0}:=Y$ in this proof).
By assumption and $Y_{2}\in [\Sub(T)]_{k}$, we have $\Gh^{k}_{T}(-,Y_{2})=0$.
Since $f_{1}$ is  $T$-coghost, we obtain that $g_{2}\circled f_{1}$ is also a $T$-coghost by Remark \ref{rmk:ideal}. Then
\[g_{2}\circled (f_{1} \circled f_{2} \circled \cdots \circled f_{k})=(g_{2}\circled f_{1}) \circled f_{2} \circled \cdots \circled f_{k}\in\Gh^{k}_{T}(-,Y_{2}).\]
Moreover, $g_{2}\circled (f_{1} \circled f_{2} \circled \cdots \circled f_{k})=0$.
Then, since $g_1$ is the kernel of $g_2$, there exists a unique morphism $h:X_{k}\to Y_{1}$ such that
\[f_{1} \circled f_{2} \circled \cdots \circled f_{k}=g_{1}\circled h,\]
that is, we obtain the following diagram
\[\xymatrix{
& T_1 &&&\\
  0 \ar[r]&Y_{1}\ar@{^(->}[u]_{h_{1}}\ar[r]^-{g_{1}}
& Y\ar[r]^-{g_{2}}
& Y_{2}\ar[r]
& 0
& \\
&& X_{1} \ar[u]^{f_{1}} &&&  \\
&& \vdots\ar[u]^{f_{2}} & & &  \\
&& X_{k}\ar[u]^{f_{k}}\ar@{-->}@/^1pc/[luuu]^{h} & & &  \\
&& X_{k+1} \ar[u]^{f_{k+1}} \ar@/_3pc/[uuuu]_{f}
& & &   }\]
commutes.
Thus, $h_{1}\circled (h \circled f_{k+1})=(h_{1}\circled h )\circled f_{k+1}=0$ by using $f_{k+1}$ to be a $T$-coghost. It follows that $h \circled f_{k+1}=0$.
Then $f=f_{1} \circled \cdots \circled f_{k+1}=g_{1}\circled h \circled f_{k+1}=g_{1}\circled 0=0$, i.e., $\Gh_{T}^{k+1}(-,Y)=0.$

(b) It is trivial that (2) implies (1) by (a). Next, we prove (1) implies (2).

First of all, we show that (1) implies (2) in the case of $n=1$.
Assume that $Y$ is an object satisfying $\Gh^{1}_{T}(-,Y)=0$,
and $f: Y\to T_1$ is a left $\add T$-approximation of $Y$.
Then for each map $h:Y\to T_{1}'$ with $T_{1}'\in\add T$,
there is a morphism $g:T_1\to T_1'$ such that the following diagram
\[\xymatrix{
 \Ker f\ar@{^(->}[r]^{f_{1}}
& Y\ar[r]^-{f}\ar[d]_-{h}
& T_{1} \ar@{-->}[ld]^-{g}
& \\
& T_{1}' &&  \\
}\]
commutes. Thus, $h\circled f_{1}=(g\circled f)\circled f_{1}=g\circled (f\circled f_{1})=g\circled 0=0$.
It follows that $f_{1}$ is a $T$-coghost.
Furthermore, $\Gh^{1}_{T}(-,Y)=0$ yields that $f_1$, as a morphism in $\Gh^{1}_{T}(\Ker f,Y)$, is zero. Thus, $\Ker f=0$.
Therefore, $f$ is monomorphic, i.e., $Y\in  [\Sub(T)]_{1}$ as required.

Next, suppose that $\Gh^{k}_{T}(-,Y)=0$.
Considering the following family of exact sequences
\begin{align}\label{formula:exact}
  (\xymatrix{0\ar[r] & Y_{i+1}\ar[r]^-{f_{i+1}} & Y_{i}\ar[r]^-{g_{i}} & T_{i}})_{0\leqslant i\leqslant k},
\end{align}
where, for all $i$, $g_i$ is a left $\add T$-approximation of $Y_{i}$;
$Y_{i+1}:=\Ker g_{i}$; and $Y_{0}:=Y$,
we know that $f_1, \ldots, f_{k+1}$ are $T$-coghosts by Lemma \ref{coghost}.
Then we obtain
\[f_{1}\circled f_{2}\circled \cdots \circled f_{k-1} \circled f_{k}=0\]
by $\Gh^{k}_{T}(-,Y)=0$, and it follows $f_k=0$ by using all $f_{i}$ ($0\leqslant i \leqslant k-1$) to be monomorphisms.
On the other hand, (\ref{formula:exact}) induces two family exact sequences
\begin{center}
$( 0 \To{}{} Y_{i+1} \To{f_{i+1}}{} Y_{i} \To{}{} \Im g_i \To{}{} 0 )_{0\leqslant i\leqslant k-1}$

and
$( 0 \To{}{} \Im g_i \To{}{} T_i)_{0\leqslant i \leqslant k-1}$
\end{center}
by the canonical decomposition of $g_i$.
It is easy to see that $\Im g_{i}\in [\Sub(T_i)]_{1} \subseteq [\Sub(T)]_{1}$ holds for each $0\leqslant i \leqslant k-1$, and then we have
\begin{align*}
Y=Y_{0}&\in [Y_{1}]_{1}\multi [ \Im g_{0}]_{1}\\
        &\subseteq [Y_{2}]_{1}\multi [\Im g_{1}]_{1}\multi [ \Im g_{0}]_{1}\\
        &\subseteq [Y_{3}]_{1}\multi [ \Im g_{2}]_{1}\multi [ \Im g_{1}]_{1}\multi [ \Im g_{0}]_{1}\\
        &\;\;\;\vdots\\
   &\subseteq [Y_{n}]_{1}\multi [ g_{n-1}]_{1}\cdots \multi [ \Im g_{2}]_{1}\multi [\Im g_{1}]_{1}\multi [\Im g_{0}]_{1}\\
   &\subseteq [0]_{1}\multi [\Sub(T)]_{1}\cdots \multi [\Sub(T)]_{1}\multi [\Sub(T)]_{1}\multi  [\Sub(T)]_{1}\\
          &=[\Sub(T)]_{n}.
\end{align*}
Therefore, (1) implies (2) by induction.
\end{proof}

\begin{remark}\label{inc-Sub} \rm
If $\mathcal{U}$ is a subcategory of $\V$,
then $[\mathcal{U}]_n \subseteq [\V]_{n}$ holds for all $n\in\NN$,
see {\rm \cite[Proposition 2.2(2)]{zheng2019extension}}.
Immediately, for each $T\in \mathcal{A}$, we have
$[T]_{n}\subseteq [\Fac(T)]_{n}$ and $[T]_{n}\subseteq [\Sub(T)]_{n}$.

\end{remark}

\subsection{Coghosts on Artin algebra}
Let $\itLamb$ be an Artin $R$-algebra defined over Artin ring $R$ in this section.
Then the category $\mod\itLamb$ of finitely generated right $\itLamb$-module (say $\itLamb$-module for short) is an Abelian category.
Now, we provide some properties of coghosts and strong generator in $\mod\itLamb$ which are dual of the properties of ghost.

\subsubsection{Some properties of coghost on Artin algebra}
Let $T$ be a finitely generated $\itLamb$-module. The following Lemma is well-known.

\begin{lemma}\label{addT-cov}
The subcategory $\add T$ of $\mod\itLamb$ is a covariantly finite subcategory.
\end{lemma}

\noindent
We will prove Proposition \ref{chain_map_coghost} through the above lemma.
Thus, for the convenience of readers, we still provide proof of the above lemma.

\begin{proof}
First of all, for each $M\in \mod\itLamb$, $\Hom_{\itLamb}(M,T)$ can be seen as a finitely generated $R$-module defined by
\[ \Hom_{\itLamb}(M,T) \times R \to \Hom_{\itLamb}(M,T),\
 (h,r) \mapsto h\cdot r,  \]
where $h\cdot r$ is the map $M\to T$ sending each $m\in M$ to $f(m)r$,
see \cite{auslander1997representation}.
Then $\Hom_{\itLamb}(M,T)$ containing a family of right $R$-homomorphism $f_1,\ldots, f_n$ which form a set of generators for $\Hom_{\itLamb}(M,T)$.

Next, take an arbitrary map $g:M\to N$ with $N\in\add T$.
Notice that there is a positive integer $m$ such that $N$ is isomorphic to a direct summand of $T^{\oplus m}$ (written as $N\le_{\oplus} T^{\oplus m}$ for simplification),
we can find two $\itLamb$-homomorphisms
\begin{center}
  $\eta = \left(\begin{smallmatrix}
  \eta_1\\ \vdots \\ \eta_m
  \end{smallmatrix}\right):
  N\to T^{\oplus m}$
  and
  $\theta = \left(
  \theta_1 \ \cdots \ \theta_m
  \right):
  T^{\oplus m} \to N$
\end{center}
such that $\theta\circled\eta = \sum_{i=1}^m \theta_i\eta_i = 1_N$.
Here, $\eta_1,\ldots, \eta_m$ are $R$-homomorphisms in $\Hom_{\itLamb}(N,T)$,
and $\theta_1,\ldots, \theta_m$ are $R$-homomorphisms in $\Hom_{\itLamb}(T,N)$.
Thus, for the $\itLamb$-homomorphism $\eta_i\circled g\in\Hom_{\itLamb}(M,T)$,
there are $n$ elements $b_{1j},\ldots, b_{nj}$ such that
$\eta_i\circled g = \sum_{j=1}^n b_{ij}f_j$. It follows that
\begin{align*}
\eta\circled g
=\left( \begin{matrix}
\eta_{1}\circled g\\
\vdots\\
\eta_{n}\circled g
\end{matrix}\right)
= \left( \begin{smallmatrix}
\sum\limits_{j=1}^{n}b_{1j}f_{j}\\
\vdots\\
\sum\limits_{j=1}^{n}b_{mj}f_{j}
\end{smallmatrix}\right)
= \left( \begin{matrix}
b_{11}&\cdots&b_{1n}\\
\vdots&      &\vdots\\
b_{m1}&\cdots&b_{mn}
\end{matrix}\right)\circled \left(
\begin{smallmatrix}
f_{1}\\
\vdots\\
f_{n}
\end{smallmatrix}\right).
\end{align*}
The right of the above equation is written as $B\circled f$ for simplification in this proof.
Then we obtain
\[g=1_{N}\circled g=(\theta \circled \eta )\circled g=\theta \circled (\eta \circled g)=\theta \circled (B\circled f)=(\theta \circled B)\circled f.\]
Thus, $g$ factors through $f$, see the following diagram
\begin{center}
\begin{tikzpicture}
\draw (0,0) node {
$\xymatrix{
& M \ar[rr]^{f}\ar[d]_{g}
&& T^{\oplus n}\ar@{-->}[lld]^{\theta\circled B}\ar@/^1pc/[llddd]^{B}\\
& N \ar[dd]^{\eta}_{\le_{\oplus}}\\
& \\
& T^{\oplus m}\ar@/^3pc/[uu]^{\theta}
}$
};
\draw (0.15,1.3) node{$\spadesuit$};
\draw (0.15,-0.2) node{$\clubsuit$};
\end{tikzpicture}
\end{center}
whose triangles $\spadesuit$ and $\clubsuit$ are commute.
Therefore, we get a $\itLamb$-homomorphism $f$ which is a left $\add T$-approximation of $M$.
Thus, $\add T$ is covariantly finite.
\end{proof}

\begin{property}\label{chain_map_coghost}
For a chain of $\add T$-coghosts
\[\xymatrix{
       Y_{n}\ar[r]^-{f_{n}} &Y_{n-1}\ar[r]^-{f_{n-1}} &\cdots \ar[r]^-{f_{3}}&Y_{2}\ar[r]^-{f_{2}} &Y_{1}\ar[r]^-{f_{1}} &Y_{0}
       }\]
in $\mod\itLamb$ satisfying $f_{1}\circled f_{2}\circled \cdots \circled f_{n}\neq 0$,
we have $Y_{0}\notin [T]_{n}.$
\end{property}

\begin{proof}
Notice that the composition $f_{1}\circled f_{2}\circled \cdots \circled f_{n}\in \Gh^{n}_{T}(Y_{n},Y_{0})$ is a non-zero $\itLamb$-homomorphism, we obtain $\Gh^{n}_{T}(Y_{n},Y_{0}) \ne 0$.
Then, by Lemma \ref{addT-cov} and Lemma \ref{coghostlemm} (b), $Y_{0}\notin [\Sub T]_{n}$ holds.
Thus, $Y_{0}\notin [T]_{n}$ since $[T]_{n}\subseteq [\Sub T]_{n}$ by Lemma \ref{inc-Sub}.
\end{proof}

\subsubsection{Some properties of strong generator over Artin algebra}
In this subdivision, we provide some properties for strong generator in the category $\mod\itLamb$ of finitely generated $\itLamb$-module over an Artin algebra $\itLamb$.

\begin{property}\label{f_is_not_coghost}
If $T$ is a strong generator of $\mod\itLamb$, then any automorphism $f:X\to X$ defined on non-zero $\itLamb$-module $X$ is not a $T$-coghost.
\end{property}

\begin{proof}
Assume $\mod\itLamb=[T]_{n}$, $n\in\NN^+$.
If $f$ is both an isomorphism and a $T$-coghost, then $f^n$ is an isomorphism which is not zero.
In thise caes, $X \notin [T]_{n}=\mod\itLamb$ by Property \ref{chain_map_coghost}, a contradiction.
\end{proof}

\begin{property} \label{lemm:cog-criterion}
Given a map $f: M_{1}\to M_{2}$ in $\mod\itLamb$.
If for each indecomposable module $N\in\add T$,
either $\Hom_{\itLamb}(M_{1},N)=0$ or $\Hom_{\itLamb}(M_{2},N)=0$,
then $f$ is a $T$-coghost.
\end{property}

\begin{proof}
Let $N' = \bigoplus_{k=1}^{t} N_k\in \add T$ be an arbitrary module,
where all direct summands $N_k$ of $N$ are indecomposable.
Then for any homomorphism $g = \left(\begin{smallmatrix} g_1 \\ \vdots \\ g_t \end{smallmatrix}\right) : M_2 \to N'$,where $g_k\in \Hom_{\itLamb}(M_2, N_k)$ for each $1\leqslant k \leqslant t$,
we have
\[ g \circled f
  = \left(\begin{smallmatrix} g_1\circled f \\ \vdots \\ g_t\circled f \end{smallmatrix}\right)
\in \Hom_{\itLamb}(M_1, N'), \]
 where $g_k\circled f\in\Hom_{\itLamb}(M_1, N_k)$.
Since either $\Hom_{\itLamb}(M_1, N_k)=0$ or $\Hom_{\itLamb}(M_2, N_k)=0$,
we obtain that $g_k\circled f=0$ for each $1\leqslant k \leqslant t$, and so $g\circled f=0$.
\end{proof}

\begin{property} \label{lemm:cog-directsum}
Given a map $f: M_{1}\to M_{2}$ in $\mod\itLamb$.
If $f$ is both $T_{1}$-coghost and $T_{2}$-coghost, then $f$ is a $(T_{1}\oplus T_{2})$-coghost.
\end{property}

\begin{proof}
For any morphism $g = \left({_{g_2}^{g_1}}\right) \in \Hom_{\itLamb}(M_2, T_1'\oplus T_2')$, where $T_1'\in\add T_1$ and $T_2'\in\add T_2$, we have
\[ g\circled f = \left({_{g_2}^{g_1}}\right) \circled f = \left({_{g_2\circled f}^{g_1\circled f}}\right). \]
Since $f$ is both $T_{1}$-coghost and $T_{2}$-coghost, we obtain $g_1\circled f=0$ and $g_2\circled f=0$.
Then $g\circled f=0$, i.e., $f$ is a $(T_{1}\oplus T_{2})$-coghost.
\end{proof}

The following lemma is well-known. We will use it to prove Proposition \ref{strong_soc_top}.

\begin{lemma}\label{soc-and-top}
Let $M$ and $N$ be two $\itLamb$-modules in $\mod\itLamb$.

    $(1)$ If $ N$ is semi-simple, then
    $\Hom_{\itLamb}(N,M)=\Hom_{\itLamb}(N,\soc M).$

    $(2)$ If $M$ is semi-simple, then
    $\Hom_{\itLamb}(N,M)=\Hom_{\itLamb}(\top N, M).$
\end{lemma}

\begin{property}\label{strong_soc_top}
If $T$ is a strong generator, then $\soc T\cong\top T\cong\itLamb/\rad\itLamb$.
\end{property}

\begin{proof}
We only show that $\soc T\cong \itLamb/\rad\itLamb$, the proof of $\top T\cong\itLamb/\rad\itLamb$ is similar by using Ghost Lemma \ref{ghostlemm}.

First of all, we show that each simple $\itLamb$-modules is a direct summand of $\soc T$.
Otherwise, there is a simple module $S\in\mod\itLamb$ such that $\Hom_{\Lambda}(S,T')=0$ holds for all $T'\in \add T$.
Then, clearly, any non-zero homomorphism $f:S\to S$ is a $T$-coghost by Proposition \ref{lemm:cog-criterion}.
It contradicts with Proposition \ref{f_is_not_coghost} because, by Schur's lemma, $f$ is an automorphism.
Thus, $S\le_{\oplus} \soc T$ holds for each simple module $S$.
It follows that $\itLamb/\rad\itLamb \cong \bigoplus_S S \le_{\oplus} \soc T$ as required.

On the other hand,  $\soc T \le_{\oplus} \bigoplus_S S$ is trivial.
Thus $\soc T \cong \itLamb/\rad\itLamb$.
\end{proof}

\begin{lemma}\label{hom_nonzero}
If $T$ is a strong generator, then $\Hom_{\itLamb}(Y,T)\ne 0$ and $\Hom_{\itLamb}(T,Y)\ne 0$ hold for all $0\ne Y\in \mod\itLamb$.
\end{lemma}

\begin{proof}
Since $f_{1}:Y\to \top Y,\,\, f_{2}:\top Y \to \soc T=\itLamb/\rad\itLamb$, and $f_{3}:\soc T \to T,$
we have $f_{3}\circ f_{2}\circ f_{1}\neq 0$. Then $\Hom_{\Lambda}(Y,T)\ne 0$.
\end{proof}

\begin{property}\label{simple_projective}
If $T$ is a strong generator of $\mod\itLamb$,
then, for arbitrary projective or injection simple modules $S_1, \ldots, S_t$,
the direct sum $\bigoplus_{i=1}^t S_i$ lies in $[T]_{1}$.
\end{property}

\begin{proof}
By Lemma \ref{hom_nonzero}, there is $\itLamb$-homomorphism $f:S\to T$ for any simple module $S$ such that $f\ne 0$. Thus, $f$ is monomorphic by Schur's lemma.
If $S$ is projective, then $f$ is a section, i.e.,
$f \simeq
\left(\begin{smallmatrix}
1_S \\ 0
\end{smallmatrix}\right): S \to S\oplus T'\ (\cong T).$
It follows that $S\le_{\oplus} T$.
We can show that $S\le_{\oplus} T$ holds in the case for $S$ to be injective by dual method.
Thus, $\bigoplus_{i=1}^t S_i \in \add T$ $= [T]_1$ since $\add T$ is closed under direct sum.
\end{proof}

\begin{property}
If $T$ is a strong generator in $\mod\itLamb$,
then, for any $T$-coghost $f:X\to Y$ with $Y$ indecomposable, we have $f\in\rad_A(X,Y)$.
\end{property}

\begin{proof}
If $f\notin \rad_{A}(X,Y)=\{f\in\Hom_{A}(X,Y) \mid 1_{Y}-f\circled g\text{ is an isomorphism for all }g\in\Hom_{A}(Y,X)\}$, then
there exists a morphism $g\in \Hom_{A}(Y,X)$ such that $1_{Y}-fg$ is not an isomorphism.
Since $f$ is a $T$-coghost, we have
$h\circled f =0$ for all $h\in \Hom_{A}(Y,T')$ with $T'\in \add T.$ Then
$$h\circled (1_{Y}-f\circled g)=h\circled 1_{Y}-h\circled f\circled g=h-0\circled g=h.$$
Thus, for each $n\in\mathbb{N}^{+}$, the following equation
\begin{align*}
    h\circled (1_{Y}-f\circled g)^{n}=&h\circled (1_{Y}-f\circled g)\circled (1_{Y}-f\circled g)^{n-1}\\
    =&h\circled (1_{Y}-f\circled g)^{n-1}\\
     &\vdots \\
     =&h\circled (1_{Y}-f\circled g)\\
     =&h
\end{align*}
holds. Since $1_{Y}-f\circled g$ is not an isomorphism and $Y$ is indecomposable, we get
$(1_{Y}-f\circled g)^{m}=0$ for some $m\in\mathbb{N}^{+}$.
And then we have $h=h\circled (1_{Y}-f\circled g)^{m}=h\circled 0=0$.
Thus, $\Hom_{A}(Y,T')=0$. By Lemma \ref{hom_nonzero}, we have $\Hom_{A}(Y,T')\ne 0$, this is a contradiction.
\end{proof}

\section{Loewy lengths}
The {\defines Loewy length} $\LL(M)$ of a $\itLamb$-module $M$ in $\mod\itLamb$ is defined as the length of the ascending chain
\[ 0=\rad^n M\le \rad^{n-1} M \le \cdots \le \rad^1 M \le \rad^0M =M \]
where we use ``$X \le Y$'' to represent a submodule (resp. subobject) $X$ of $Y$ in $\mod\itLamb$ (resp. $\mathcal{A}$), and, for each $1\leqslant i\leqslant n$, $\rad^i M$ is the submodule of $M$ satisfying $\rad^i M=\rad(\rad^{i-1}M)$.
Obviously, the Loewy length of any semi-simple module is $1$, and the Loewy length of a $\itLamb$-module $M$ is zero if and only if $M=0$.

\subsection{Properties of Loewy lengths}
We recall some well-known results for Loewy length in this subsection.

\begin{lemma} \label{lemm:Loewy length 1}
For arbitrary short exact sequence
\[0 \To{}{} X \To{f}{} Y \To{g}{} Z \to 0\]
in $\mod\itLamb$, we have
\[\max\{\LL(Y),\LL(Z)\}\leqslant\LL(Y)\leqslant\LL(X)+\LL(Z).\]
\end{lemma}

\begin{lemma}\label{lemm:Loewy length 2}
For each $M\in [T]_{m}$, we have $\LL(M)\leqslant m\LL(T)$.
\end{lemma}

\begin{proof}
In the case of $m=1$, it is trivial that if $M\in[T]_1=\add T$ then $\LL(M)\leqslant\LL(T)$.
Consider the induction hypothesis which admit that if $m\le t$ then $\LL(M)\leqslant t\LL(T)$.
Then for $m=t+1$, there is a short exact sequence
\[0 \To{}{} A \To{}{} M \To{}{} B \To{}{} 0\]
such that $A\in [T]_1$ and $B\in [T]_t$ hold for arbitrary $M\in [T]_{t+1}$.
Then, by Lemma \ref{lemm:Loewy length 1}, we have
\[ \LL(M) \leqslant \LL(A)+\LL(B) \leqslant 1\LL(T) + t\LL(T) = (1+t)\LL(T). \]
Thus, this lemma holds by induction.
\end{proof}

Notice that if $T$ is a strong generator of $\mod\itLamb$
and the $m$ is the positive integer satisfying $[T]_m=\mod\itLamb$,
then by Lemma \ref{lemm:Loewy length 2}, the following two statement hold.
\begin{itemize}
  \item[\rm(1)] $\ell\ell(\itLamb)\leqslant m\ell\ell(T)$;
  \item[\rm(2)] for each $M\in [\itLamb/\rad\itLamb]_{m}$, we have $\LL(M)\leqslant m\LL(\itLamb/\rad\itLamb) = m$.
\end{itemize}
The statement (1) holds by $\itLamb\in[T]_m$; and the statement (2) can be shown in the case for $T=\itLamb/\rad\itLamb$.

\begin{proposition}\label{prop:GentTime}
$\gent_{\mod\itLamb}(\itLamb/\rad\itLamb)=\LL(\itLamb)-1.$
\end{proposition}

\begin{proof}
The case $\LL(\itLamb)=1$ is trivial. Now consider the
case $\LL(\itLamb)\geqslant 2.$
For each $M\in \mod\itLamb,$
the following family of short exact sequence
\[ (0 \To{}{} \rad^t M \To{}{} \rad^{t-1}M \To{}{} \rad^{t-1}M/\rad^t M \To{}{} 0 )_{1\leqslant t\leqslant \LL(\itLamb)-1} \]
satisfies $\rad^{t-1}M/\rad^t M \in [\itLamb/\rad\itLamb]_1$ for all $1\leqslant t\leqslant \LL(\itLamb)-1$
since all $\rad^{t-1}M/\rad^t M$ are semi-simple.
Then $[\itLamb/\rad\itLamb]_{\LL(\itLamb)} = \mod\itLamb$.
It follows that $\gent_{\mod\itLamb}(\itLamb/\rad\itLamb)\leqslant\LL(\itLamb)-1$.

On the other hand, if $\gent_{\mod\itLamb}(\itLamb/\rad\itLamb)\leqslant\LL(\itLamb)-2$,
then $\itLamb \in \mod\itLamb = [\itLamb/\rad\itLamb]_{\LL(\itLamb)-1}$.
It follows $\LL(\itLamb)\leqslant (\LL(\itLamb)-1)\LL(\itLamb/\rad\itLamb)$
by Lemma \ref{lemm:Loewy length 2}, a contradiction.
\end{proof}


\begin{corollary}\label{coro:LL-udim}
The following two statements hold.
\begin{itemize}
  \item[\rm(1)] $\LL(\itLamb)-1\in \ospec(\itLamb).$
  \item[\rm(2)] $\udim(\itLamb) \geqslant \LL(\itLamb)-1.$
\end{itemize}
\end{corollary}

\begin{proof}
By Proposition \ref{prop:GentTime}, we have
$\gent_{\mod\itLamb}(\itLamb/\rad\itLamb) = \LL(\itLamb)-1$
holds for arbitrary Artin algebra $\itLamb$. Then the statement (1) holds by Definition \ref{def:OSpec}.
Furthermore, the statement (2) is obtained by (1) and the definition of ultimate dimension ( see Definition \ref{def:OSpec}) as following formula
\[ \udim(\itLamb) = \sup(\ospec(\itLamb)) \geqslant \LL(\itLamb)-1 \]
as required.
\end{proof}

\subsection{Strong generators}
Recall that the following well-known lemma which will be used to prove Lemma \ref{lemm:pdM=<pdT} in our paper.

\begin{lemma}\label{lemm:ext-pdim}
For any short exact sequence
\[ 0 \To{}{} A \To{}{} B \To{}{} C \To{}{} 0 \]
in $\mod\itLamb$, we have
\begin{center}
$\pd B\leqslant \pd(A\oplus C)$ and $\id B\leqslant\id(A\oplus C)$
\footnote{For any module $X$ in $\mod\itLamb$, ``$\pd X$'' and ``$\id X$'' are projective dimension and injective dimension of $X$, respectively. }
\end{center}
\end{lemma}

\begin{lemma}\label{lemm:pdM=<pdT}
Let $M$ and $T$ be two $\itLamb$-modules such that $M\in[T]_{n+1}$.
Then $\pd M\leqslant \pd T$ and $\id M\leqslant \id T$.
\end{lemma}

\begin{proof}
First of all, by $M\in [T]_{n+1}$, there is the following family of short exact sequences
\[ (0 \To{}{} X_i \To{}{} Y_{i-1} \To{}{} Z_i \To{}{} 0)_{1\leqslant i\leqslant n} \]
such that $Z_{i-1}\le_{\oplus}Y_{i-1}$ (in this case, $Y_{i-1}\cong Z_{i-1}\oplus Z_{i-1}'$ for some $Z_{i-1}'$),
where $Z_0=M$; $X_1,\ldots, X_n \in [T]_1$; and $Z_i \in [T]_{n-i+1}$.
Then, naturally, we have the following short exact sequences
\[ 0
\To{}{} \bigg(\bigoplus_{t=1}^{i} X_t\bigg) \oplus X_{i+1}
\To{}{} \bigg(\bigoplus_{t=1}^{i} X_t\bigg) \oplus Y_{i}
\To{}{} Z_{i+1}
\To{}{} 0 \]
for any $1\leqslant i\leqslant n-1$.
Then, by Lemma \ref{lemm:ext-pdim} and the above two families of short exact sequences, we have
\begin{align*}
  \pd M & = \pd Z_0 \leqslant \pd Y_0 \leqslant \pd(X_1\oplus Z_1) \\
& \leqslant \pd(X_1\oplus Z_1 \oplus Z_1') = \pd(X_1\oplus Y_1) \\
& \leqslant \pd(X_1\oplus X_2 \oplus Z_2) \\
& \leqslant \pd(X_1\oplus X_2 \oplus Z_2 \oplus Z_2')
  = \pd(X_1\oplus X_2 \oplus Y_2) \\
& \leqslant \pd(X_1\oplus X_2\oplus X_3\oplus Z_3) \\
& \leqslant \pd(X_1\oplus X_2\oplus X_3\oplus Z_3 \oplus Z_3')
  = \pd(X_1\oplus X_2\oplus X_3\oplus Y_3) \\
&  \;\;\;\;\;\vdots \\
& \leqslant \pd\bigg(\bigg(\bigoplus_{t=1}^n X_t\bigg)\oplus Y_n\bigg)
  \mathop{\leqslant}\limits^{\spadesuit} \pd T,
\end{align*}
where $\spadesuit$ holds since all $X_i$ and $Y_n$ are $\itLamb$-modules lying in $[T]_1=\add T$.
We can prove $\id M \leqslant \id T$ by similar way.
\end{proof}

\begin{proposition} \label{prop:pdT=gldimLamb}
If $T$ is a strong generator of $\mod\itLamb$, then the global dimension $\gldim\itLamb$ of $\itLamb$ and the projective dimension $\pd T$ of $T$ coincide.
\end{proposition}

\begin{proof}
Assume $\mod\itLamb=[T]_{n+1}$. Then we have $\pd(\itLamb/\rad\itLamb) \leqslant \pd T$
by Lemma \ref{lemm:pdM=<pdT} since it is clear that $\itLamb/\rad\itLamb$ is a $\itLamb$-module in $\mod\itLamb = [T]_{n+1}$.
It follows that $\gldim\itLamb = \pd(\itLamb/\rad\itLamb) \leqslant \pd T$.
On the other hand, $\pd T \leqslant \gldim \itLamb$ by the definition of global dimension of the Artin algebra $\itLamb$.
Thus, $\pd T=\gldim \itLamb$ as required.
\end{proof}

\begin{corollary} \label{coro:pdT=gldimLamb}
Keep the notations from Proposition \ref{prop:pdT=gldimLamb}, then the following two statements hold.
\begin{itemize}
  \item[\rm(1)] There is an indecomposable module $M$ with
    $\pd M=\gldim \itLamb$ such that $M\in [T]_{1}$.

  \item[\rm(2)] There is an indecomposable module $N$ with
    $\id N=\gldim \itLamb$ such that $N\in [T]_{1}$.
\end{itemize}
\end{corollary}

\begin{proof}
We only prove (1), the proof of (2) is similar.
Considering the decomposition $T = \bigoplus_{i=1}^{t} T_i$ of $T$, where $T_1,\ldots, T_t$ are indecomposable,
we obtain \[\gldim\itLamb = \pd T = \max_{1\leqslant i\leqslant t} \pd T_i,\]
that is, there is an index $\imath$ such that $\pd T = \pd T_{\imath}$ as required.
\end{proof}

Next we provide an instance for Corollary \ref{coro:pdT=gldimLamb}.

\begin{example} \rm
Let $\itLamb=\kk\Q/\I$ be the basic algebra whose quiver and admissible ideal is given by
\[ 1 \To{a}{} 2 \To{b}{} 3 \ \text{ and }\ \I=\langle ab\rangle, \]
respectively. One can check that $T:=S(1)\oplus S(2)\oplus S(3)$ is a strong generator of the finitely generated module category $\mod\itLamb$ of $\itLamb$ and the projective dimension of the direct summand $S(1)$ of $T$ equals to $2$.
Moreover, we have $\pd S(2)=1$ and $\pd S(3)=0$.
Therefore, $\pd T = \pd S(1) = \gldim\itLamb = 2$ in this example.
\end{example}

\subsection{Loewy lengths of some Artin alegbras}

In this subsection we consider the Loewy lengths of some Artin algebras. This section containing three parts.
In the parts one, i.e., \ref{subsubsect:simp-proj-inj}, we show that the Loewy length of an {\defines {\rm SPI} Artin algebra},
an Artin algebra whose simple modules are either projective or injective, is less than or equal to two.
In the parts two, i.e., \ref{subsubsect:semisimp}, we provide a description of semi-simple algebra by using Loewy length.
In the last parts, i.e., \ref{subsubsect:spec=0,1}, we prove that the Loewy length of an Artin algebra with Orlov spectrum $\{0,1\}$ equals to two,
and compute the Orlov spectrum of any SPI Artin algebra.

\subsubsection{Loewy lengths of {\rm SPI} Artin algebras} \label{subsubsect:simp-proj-inj}
\begin{proposition}\label{prop:LL=2}
If $\itLamb$ is an {\rm SPI} Artin algebra, then $\LL(\itLamb)\leqslant 2$.
\end{proposition}

\begin{proof}
Assume that there is an SPI Artin algebra $\itLamb$ such that all simple modules are either projective or injective and $\LL(\itLamb)=n\geqslant 3$.
Then there is a projective module $P$ with $\LL(P)=n$.
In particular, we have
$\LL(\rad P)=n-1$ and $\LL(\rad^{2} P)=n-2$.
The following short exact sequence
\[0 \longrightarrow \rad^{2} P \longrightarrow  \rad P\longrightarrow  \rad P/\rad^2 P \longrightarrow 0\]
admits $\rad P/\rad^2 P = Q\oplus E$ such that $Q$ is the sum of some simple projective modules and $E$ is the sum of some simple injective modules by our assumption.
Then the pullback of
\[ \xymatrix@C=2cm{
  & Q \ar[d]^{\le_{\oplus}} \\
\rad P \ar[r]_{
\begin{smallmatrix}
\text{canonical} \\ \text{epimorphism}
\end{smallmatrix}
}
  & \rad P/\rad^2 P
} \]
is of the following form.
\begin{gather}
\begin{split}
\xymatrix{
          &                     & 0 \ar[d]                 &   0       \ar[d]                        &  \\
       0\ar[r]   &  \rad^{2} P \ar[r]\ar@{=}[d] & W\ar[r]\ar[d]&     {\color{red}Q}\ar[r]\ar[d]     & 0  \\
0\ar[r]&\rad^{2} P \ar[r]    & {\color{red}\rad P }\ar[r]\ar[d]&{\color{red} \rad P/\rad^{2} P}\ar[d]\ar[r]&0  \\
& & E \ar@{=}[r]\ar[d]& E\ar[d]&  \\
           &                                 &0                             & 0                                  & }
\end{split}.
\end{gather}
Since $Q$ is projective and $E$ is injective, we have
$W\cong \rad^{2} P\oplus Q$ and $\rad P \cong W \oplus E$.
Thus, we have
\begin{align*}
  n-1 & =\LL(\rad P) = \LL(W\oplus E) = \LL(\rad^{2} P\oplus Q\oplus E) \\
   &=\max\{ \LL(\rad^{2} P) , \LL(Q), \LL(E)\} = n-2
\end{align*}
by $n\geqslant 3$, a contradiction.
\end{proof}

\begin{example}\rm
The inverse of Proposition \ref{prop:LL=2} maybe not hold. For example, let the path algebra
$\itLamb=\kk\Q/\langle \alpha\beta,\beta\alpha\rangle$,
where the quiver $Q$ as follows
$$\xymatrix{
  1\ar@/^/[r]^{\alpha}&2\ar@/^/[l]^{\beta}.
    }$$
In this case, we have $\LL(\itLamb)=2$, but $S(1)$ and $S(2)$ are neither projective nor injective.
\end{example}

\subsubsection{Loewy lengths of semi-simple Artin algebras} \label{subsubsect:semisimp}

\begin{lemma}\label{lemm:0-in-ospec}
Assume that $\itLamb$ is a representation-finite Artin algebra which is not semi-simple. Then $\{0,1\} \subseteq \ospec(\itLamb)$.
\end{lemma}

\begin{proof}
Let $\mathsf{ind}(\mod\itLamb)$ be the set of all indecomposable $\itLamb$-modules (up to isomorphism).
Since $\itLamb$ is representation-finite, that is, $\mathsf{ind}(\mod\itLamb)$ is a finite set, it is easy to see that
\[\displaystyle T=\bigoplus_{X\in\mathsf{ind}(\mod\itLamb)} X\]
is a strong generator of $\mod\itLamb$, and, in this case, we obtain $[T]_1=\mod\itLamb$.
Thus, $0\in\ospec(\itLamb)$.

On the other hand, there is an indecomposable projective module $P$ whose Loewy length $\LL(P)$ equals to the Loewy length $\LL(\itLamb)$ of $\itLamb$. Let
\[T' = \bigoplus_{X\in \mathsf{ind}(\mod\itLamb)\backslash\{P\}} X, \]
then the following short exact sequence
\[ 0 \To{}{} \rad P \To{}{} P \To{}{} P/\rad P \To{}{} 0 \]
shows that $P \in [T']_2$ since $\rad P,P/\rad P\in [T']$. On the other hand,
$P\notin [T']_1.$
That is, $[T']_{1}\subsetneq [T']_{2}=\mod \itLamb$.
It follows $1\in\ospec(\itLamb)$.
\end{proof}

\begin{theorem}\label{thm:semi-simp}
For an Artin algebra $\itLamb$, the following statement are equivalent.
\begin{itemize}
  \item[\rm(1)] $\itLamb$ is semi-simple.
  \item[\rm(2)] $\LL(\itLamb)=1$.
  \item[\rm(3)] $\ospec(\itLamb)=\{0\}$.
\end{itemize}
\end{theorem}

\begin{proof}
First of all, since $\LL(\itLamb)=1$, we obtain $\LL(\itLamb/\rad\itLamb)=0$.
It follows $\itLamb/\rad\itLamb=0$, i.e., $\itLamb$ is semi-simple. Thus, (2) admits (1).

Second, if $\itLamb$ is semi-simple, then all indecomposable $\itLamb$-modules are simple modules which are direct summands of $\itLamb$.
It is easy to see that $\LL(\itLamb)=1$, Thus, (1) and (2) are equivalent.

Third, we show that (1) admits (3). Notice that semi-simple Artin algebra are always representation-finite.
Thus, we can assume that $\itLamb=\bigoplus_{i=1}^{n} S_i$ is basic up to Morita equivalence,
where $S_1,\ldots S_n$ are simple.
Then we obtain that each basic $\itLamb$-module $M$ is isomorphic to
\[ \bigoplus_{i=1}^n S_i^{\oplus t_i},\ \ (t_i\in\{0,1\}) \]
which lies in $[\itLamb]_1=\mod\itLamb$. Thus, $0\in\ospec(\itLamb)$.
However, if a $\itLamb$-module $T$ does not contain some simple module $S_i$ as its direct summand (up to isomorphism),
then $[T]_{n}\subsetneq\mod\itLamb$ holds for all $n\in\NN$. Then $\ospec(\itLamb)=\{0\}$.

Finally, it is easy to prove that $\itLamb$ is semi-simple if $\ospec(\itLamb)=\{0\}$ by Lemma \ref{lemm:0-in-ospec}. Therefore, the statements (1), (2) and (3) are equivalent.
\end{proof}

\subsubsection{Loewy length of Artin algebras with Orlov spectra $\{0,1\}$} \label{subsubsect:spec=0,1}

The following result shows that the Loewy length of an Artin algebra with Orlov spectrum $\{0,1\}$ equals to two.

\begin{proposition} \label{prop:OSpec=(0,1)}
If $\ospec(\itLamb) = \{0,1\}$, then $\LL(\itLamb)=2$.
\end{proposition}

\begin{proof}
If $\LL(\itLamb)=n \geqslant 3$, then $n\in\ospec(\itLamb)$ by Proposition \ref{prop:GentTime} (or Corollary \ref{coro:LL-udim} (1)), a contradiction. Thus, $\LL(\itLamb)\leqslant 2$.
By Theorem \ref{thm:semi-simp} and $\ospec(\itLamb) = \{0,1\}$, we obtain that $\itLamb$ is not semi-simple, i.e., $\LL(\itLamb)\geqslant 2$.
Thus, $\LL(\itLamb)=2$.
\end{proof}

\begin{theorem}\label{thm:non semi-simp SPI}
Let $\itLamb$ be an {\rm SPI} Artin algebra which is not semi-simple.
\begin{itemize}
  \item[\rm(1)] If $\itLamb$ is representation-finite,
    then $\ospec(\itLamb)=\{0,1\}.$

  \item[\rm(2)] If $\itLamb$ is representation-infinite,
    then $\ospec(\itLamb)=\{1\}.$
\end{itemize}
\end{theorem}

\begin{proof}
Assume $\mod\itLamb=[T]_{n+1}$ for some $n\in\NN$. By Property \ref{simple_projective},
all simple modules lie in $[T]_1=\add T$ since every simple module is either a projective module or an injective module.
Thus, $T \cong (\itLamb/\rad\itLamb)\oplus T'$ for some module $T'$.
Then we obtain
\begin{align*}
\gent_{\text{mod}\itLamb}(T)&=\gent_{\text{mod}\itLamb}((\itLamb/\rad\itLamb)\oplus T') \\
    &\leqslant\gent_{\text{mod}\itLamb}(\itLamb/\rad\itLamb) \tag{\text{by Lemma } \ref{time1}} \\
    &=\LL(\itLamb)-1
     \tag{\text{Proposition } \ref{prop:GentTime}} \\
    &=2-1 = 1. \tag{\text{Proposition } \ref{prop:LL=2}}
\end{align*}
It follows that $\udim(\mod\itLamb) = \sup(\ospec(\itLamb))\leqslant 1$.
On the other hand,
$\udim(\mod\itLamb) \geqslant \LL(\itLamb)-1 = 1$ by Corollary \ref{coro:LL-udim} (2).
That is, $\udim(\mod\itLamb) = 1$. Thus, $1\in\ospec(\itLamb)$.

If $\itLamb$ is representation-finite, then,
by Lemma \ref{lemm:0-in-ospec},
we have $\ospec(\itLamb)=\{0,1\}$, that is, the statement (1) holds;
otherwise, $\ospec(\itLamb)=\{1\}$, that is, the statement (2) holds.
\end{proof}

Next, we provide some examples for Theorem \ref{thm:non semi-simp SPI}.

\begin{example} \rm
Consider the hereditary algebra $\itLamb=\vec{A}_2=\kk\Q$ with $\Q= 1\To{\alpha}{} 2$.
It is representation-finite, to be more precise, the number of isoclasses of indecomposable module equals to $3$.
Since the simple module $S(2)$ is projective and the simple module $S(1)$ is injective,
we have $\ospec(\itLamb)=\{0,1\}$ by Theorem \ref{thm:non semi-simp SPI} (1).
\end{example}

\begin{example} \label{exp:non semi-simp SPI} \rm
(1) Consider the path algebra $\itLamb=\kk\Q$ given by the following quiver shown in Figure \ref{fig:quiver_in_Example240623}.
\begin{figure}[htbp]
\centering
\centering
\begin{tikzpicture}[scale=1]
\draw[line width=1pt][rotate=  0][->] (3.25,-0.30) -- ( 0.20,-0.01);
\draw[line width=1pt][rotate= 45][->] (3.25,-0.30) -- ( 0.20,-0.01);
\draw[line width=1pt][rotate= 90][->] (3.25,-0.30) -- ( 0.20,-0.01);
\draw[line width=1pt][rotate=135][->] (3.25,-0.30) -- ( 0.20,-0.01);
\draw[line width=1pt][rotate=180][->] (3.25,-0.30) -- ( 0.20,-0.01);
\draw[line width=1pt][rotate=225][->] (3.25,-0.30) -- ( 0.20,-0.01) [dotted];
\draw[line width=1pt][rotate=270][->] (3.25,-0.30) -- ( 0.20,-0.01) [dotted];
\draw[line width=1pt][rotate=315][->] (3.25,-0.30) -- ( 0.20,-0.01);
\draw[rotate=  0] (1.71, 0.05) node {\small{$\alpha_1$}};
\draw[rotate= 45] (1.71, 0.05) node {\small{$\alpha_2$}};
\draw[rotate= 90] (1.71, 0.05) node {\small{$\alpha_3$}};
\draw[rotate=135] (1.71, 0.05) node {\small{$\alpha_4$}};
\draw[rotate=180] (1.71, 0.05) node {\small{$\alpha_5$}};
\draw[rotate=315] (1.71, 0.05) node {\small{$\alpha_{n-1}$}};
\draw[rotate=  0] (3.45,-0.31) node{\small\rotatebox{0}{$1$}};
\draw[rotate=  0] (0,0) node{\small\rotatebox{0}{$0$}};
\draw[rotate= 45] (3.45,-0.31) node{\small\rotatebox{0}{$2$}};
\draw[rotate= 90] (3.45,-0.31) node{\small\rotatebox{0}{$3$}};
\draw[rotate=135] (3.45,-0.31) node{\small\rotatebox{0}{$4$}};
\draw[rotate=180] (3.45,-0.31) node{\small\rotatebox{0}{$5$}};
\draw[rotate=225] (3.45,-0.31) node{\small\rotatebox{0}{$\cdots$}};
\draw[rotate=270] (3.45,-0.31) node{\small\rotatebox{0}{$\cdots$}};
\draw[rotate=315] (3.45,-0.31) node{\small\rotatebox{0}{$n$}};
\end{tikzpicture}
\caption{The path algebra in Example \ref{exp:non semi-simp SPI} (1)}
\label{fig:quiver_in_Example240623}
\end{figure}
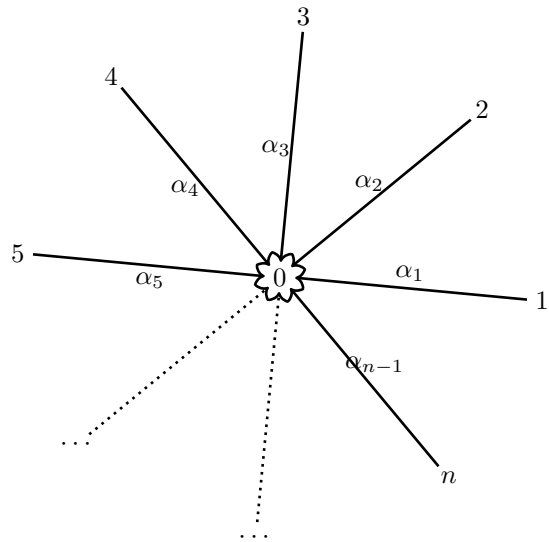
It is clearly that $\itLamb$ is a hereditary algebra,
and then $\itLamb$ is representation-infinite since
the underlying graph of $\Q$ is not a Dynkin type
(cf. \cite[Chap VII, Section 2]{assem2006elements}).
Moreover, the simple module $S(0)$ is projective
and, for any $1\leqslant i\leqslant n$, the simple module $S(i)$ is injective.
Therefore, $\ospec(\Lambda)=\{1\}$ by Theorem \ref{thm:non semi-simp SPI} (2).

(2) Consider the path algebra $\Lambda=\kk\Q$ given by the quiver shown in Figure \ref{fig:path alg}, where $t$ is an even number with $t\geqslant 4$.
Now we compute its Orlov spectrum.
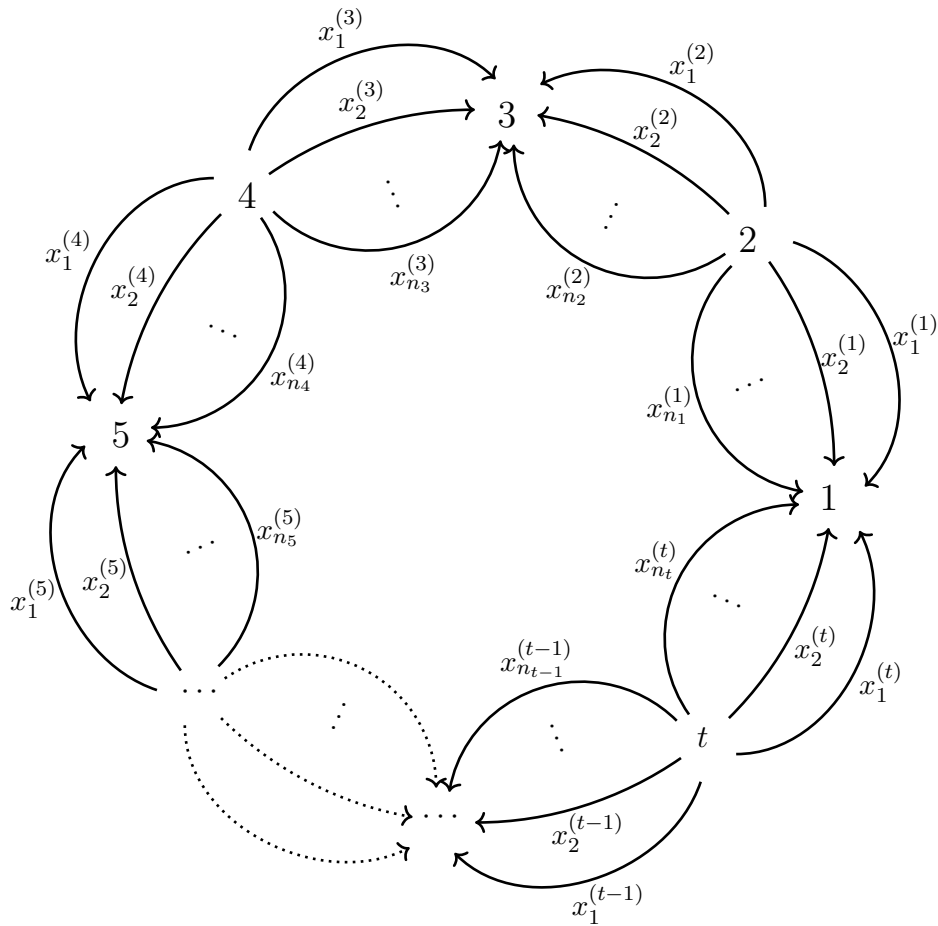
\begin{figure}[htpb]
\centering
\begin{tikzpicture}[scale=1.35]
\draw[<-][line width=1pt][rotate=  0] (3.5,0) arc (0:35:3.5);
\draw[->][line width=1pt][rotate= 45] (3.5,0) arc (0:35:3.5);
\draw[<-][line width=1pt][rotate= 90] (3.5,0) arc (0:35:3.5);
\draw[->][line width=1pt][rotate=135] (3.5,0) arc (0:35:3.5);
\draw[<-][line width=1pt][rotate=180] (3.5,0) arc (0:35:3.5);
\draw[->][line width=1pt][rotate=225] (3.5,0) arc (0:35:3.5)[dotted];
\draw[<-][line width=1pt][rotate=270] (3.5,0) arc (0:35:3.5);
\draw[->][line width=1pt][rotate=315] (3.5,0) arc (0:35:3.5);
\draw[rotate=  0] (3.45,-0.31) node{\Large\rotatebox{0}{$1$}};
\draw[rotate= 45] (3.45,-0.31) node{\Large\rotatebox{0}{$2$}};
\draw[rotate= 90] (3.45,-0.31) node{\Large\rotatebox{0}{$3$}};
\draw[rotate=135] (3.45,-0.31) node{\Large\rotatebox{0}{$4$}};
\draw[rotate=180] (3.45,-0.31) node{\Large\rotatebox{0}{$5$}};
\draw[rotate=225] (3.45,-0.31) node{\large\rotatebox{0}{$\cdots$}};
\draw[rotate=270] (3.45,-0.31) node{\large\rotatebox{0}{$\cdots$}};
\draw[rotate=315] (3.45,-0.31) node{\large\rotatebox{0}{$t$}};
\draw[<-][line width=1pt][rotate=  0] (3.8,-0.2) to[out=45,in=-18.5] (3.1,2.2);
\draw[->][line width=1pt][rotate= 45] (3.8,-0.2) to[out=45,in=-18.5] (3.1,2.2);
\draw[<-][line width=1pt][rotate= 90] (3.8,-0.2) to[out=45,in=-18.5] (3.1,2.2);
\draw[->][line width=1pt][rotate=135] (3.8,-0.2) to[out=45,in=-18.5] (3.1,2.2);
\draw[<-][line width=1pt][rotate=180] (3.8,-0.2) to[out=45,in=-18.5] (3.1,2.2);
\draw[->][line width=1pt][rotate=225] (3.8,-0.2) to[out=45,in=-18.5] (3.1,2.2)[dotted];
\draw[<-][line width=1pt][rotate=270] (3.8,-0.2) to[out=45,in=-18.5] (3.1,2.2);
\draw[->][line width=1pt][rotate=315] (3.8,-0.2) to[out=45,in=-18.5] (3.1,2.2);
\draw[<-][line width=1pt][rotate=  0] (3.2,-0.25) arc(-100:-225:1.31);
\draw[->][line width=1pt][rotate= 45] (3.2,-0.25) arc(-100:-225:1.31);
\draw[<-][line width=1pt][rotate= 90] (3.2,-0.25) arc(-100:-225:1.31);
\draw[->][line width=1pt][rotate=135] (3.2,-0.25) arc(-100:-225:1.31);
\draw[<-][line width=1pt][rotate=180] (3.2,-0.25) arc(-100:-225:1.31);
\draw[->][line width=1pt][rotate=225] (3.2,-0.25) arc(-100:-225:1.31)[dotted];
\draw[<-][line width=1pt][rotate=270] (3.2,-0.25) arc(-100:-225:1.31);
\draw[->][line width=1pt][rotate=315] (3.2,-0.25) arc(-100:-225:1.31);
\draw[rotate=  0] (4.3,1.3) node{\rotatebox{0}{$x_{1}^{(1)}$}};
\draw[rotate=  0] (3.6,1.1) node{\rotatebox{0}{$x_{2}^{(1)}$}};
\draw[rotate=  0] (1.9,0.6) node{\rotatebox{0}{$x_{n_1}^{(1)}$}};
\draw[rotate= 45] (4.3,1.3) node{\rotatebox{0}{$x_{1}^{(2)}$}};
\draw[rotate= 45] (3.6,1.1) node{\rotatebox{0}{$x_{2}^{(2)}$}};
\draw[rotate= 45] (1.9,0.6) node{\rotatebox{0}{$x_{n_2}^{(2)}$}};
\draw[rotate= 90] (4.3,1.3) node{\rotatebox{ 0}{$x_{1}^{(3)}$}};
\draw[rotate= 90] (3.6,1.1) node{\rotatebox{ 0}{$x_{2}^{(3)}$}};
\draw[rotate= 90] (1.9,0.6) node{\rotatebox{0}{$x_{n_3}^{(3)}$}};
\draw[rotate=135] (4.3,1.3) node{\rotatebox{0}{$x_{1}^{(4)}$}};
\draw[rotate=135] (3.6,1.1) node{\rotatebox{0}{$x_{2}^{(4)}$}};
\draw[rotate=135] (1.9,0.6) node{\rotatebox{0}{$x_{n_4}^{(4)}$}};
\draw[rotate=180] (4.3,1.3) node{\rotatebox{0}{$x_{1}^{(5)}$}};
\draw[rotate=180] (3.6,1.1) node{\rotatebox{0}{$x_{2}^{(5)}$}};
\draw[rotate=180] (1.9,0.6) node{\rotatebox{0}{$x_{n_5}^{(5)}$}};
\draw[rotate=270] (4.3,1.3) node{\rotatebox{0}{$x_{1}^{({t-1})}$}};
\draw[rotate=270] (3.6,1.1) node{\rotatebox{0}{$x_{2}^{({t-1})}$}};
\draw[rotate=270] (1.9,0.6) node{\rotatebox{0}{$x_{n_{t-1}}^{({t-1})}$}};

\draw[rotate=315] (4.3,1.3) node{\rotatebox{0}{$x_{1}^{({t})}$}};
\draw[rotate=315] (3.6,1.1) node{\rotatebox{0}{$x_{2}^{({t})}$}};
\draw[rotate=315] (1.9,0.6) node{\rotatebox{0}{$x_{n_{t}}^{({t})}$}};



\draw[rotate=  0] (2.7,0.8) node{\rotatebox{ 22.5}{$\cdots$}};
\draw[rotate= 45] (2.7,0.8) node{\rotatebox{ 67.5}{$\cdots$}};
\draw[rotate= 90] (2.7,0.8) node{\rotatebox{112.5}{$\cdots$}};
\draw[rotate=135] (2.7,0.8) node{\rotatebox{157.5}{$\cdots$}};
\draw[rotate=180] (2.7,0.8) node{\rotatebox{202.5}{$\cdots$}};
\draw[rotate=225] (2.7,0.8) node{\rotatebox{247.5}{$\cdots$}};
\draw[rotate=270] (2.7,0.8) node{\rotatebox{292.5}{$\cdots$}};
\draw[rotate=315] (2.7,0.8) node{\rotatebox{337.5}{$\cdots$}};
\end{tikzpicture}
\caption{The path algebra in Example \ref{exp:non semi-simp SPI} (2)}
\label{fig:path alg}
\end{figure}
First, it is clear that $\Lambda$ is a hereditary algebra.
Then, it is not representation finite since
the underlying graph of $\Q$ is not a Dynkin type
(cf. \cite[Chap VII, Section 2]{assem2006elements}).
On the other hand, it is easy to see that the
simple modules $S(2i)$ ($i\in\NN$, $2\leqslant 2i \leqslant t$) is indecomposable projective
and $S(2j+1)$ ($j\in\NN$, $1\leqslant 2j+1 \leqslant t$) is indecomposable injective.
Therefore, $\ospec(\Lambda)=\{1\}$ by Theorem \ref{thm:non semi-simp SPI} (2) as required.
\end{example}

\section{Radical Layer lengths}

Layer length, introduced by Huard, Lanzilotta, and Hern{\'a}ndez in \cite{huard2013layer}, is closely related to finitistic dimension.

\subsection{Torsion pairs and radical Layer lengths}
\begin{definition} \rm
A {\defines torsion pair} ({\defines torsion theory}) for $\mod\itLamb$ is a pair $(\T,\F)$
of classes of modules in $\mod\itLamb$ satisfying the following conditions:

(1) $\Hom_{\itLamb}(M,N)=0$ for each $M\in\T$ and $N\in\F$;

(2) for each $X\in\mod\itLamb$, if $\Hom_{\itLamb}(X,-)|_{\F}=0$ then $X\in\T$;

(3) for each $Y\in\mod\itLamb$, if $\Hom_{\itLamb}(-,Y)|_{\T}=0$ then $Y\in\F$.
\end{definition}

For a given class $\mathcal{X}$ in $\mod\itLamb$ and a module $M\in\mod\itLamb$,
the trace of $\mathcal{X}$ is the submodule $\text{Trace}_{\mathcal{X}}(M)$ of
$M$ which is the sum of all images $\Im(f)$ of the morphisms $f\in\Hom_{\itLamb}(X,M)$ with $X\in \mathcal{X}.$
Let $(\T,\F)$ be a torsion theory for $\mod\itLamb$. Let us recall that $t:=\text{Trace}_{\T}$
is the so called torsion radical attached to $(\T,\F)$. In fact, $t(M)$ is the largest submodule of $M$
lying in $\T$. And we write
$q_{t}:=1_{\mod\itLamb}/t$, i.e., $q_{t}(M):=M/t(M)$ for each $M\in\mod\itLamb$.
Thus, for each $M\in\mod\itLamb$, we have the following short exact sequence
\[0 \longrightarrow t(M)\longrightarrow  M\longrightarrow  q_{t}(M)\longrightarrow 0\]
where $t(M)\in \T$ and $q_{t}(M)\in\F$.

\begin{definition}[{\cite{huard2013layer}}] \label{radical-length} \rm
Given a torsion pair $(\T,\F)$ with the torsion radical $t$.
The {\defines $t$-radical layer length} is a function
$\LL^{t}: \mod\itLamb \to \NN\cup\{\infty\}$ via
\[\LL^{t}(M)=\inf\{i\geqslant 0 \mid t\circled F_{t}^{i}(M)=0, M\in \mod\itLamb\}\]
where $F_{t}=\rad\circ t$.
\end{definition}

Let $\simp$ be a subset of some simple $\itLamb$-modules and $\simp'$ the set of all the others simple $\itLamb$-modules.
A {\defines $\simp$-filtration} of a $\itLamb$-module $M$ is a filtration $0\subseteq M_{0}\subseteq M_{1} \cdots\subseteq M_{m-1}\subseteq M_m = M$ such that $M_{i}/M_{i-1}\in\simp$ holds for all $1\leqslant i\leqslant m$.
We use $\F_{\simp}$ to represent the subcategory of $\mod\itLamb$ generated by all $\itLamb$-modules $M$ with some $\simp$-filtration.
We have the following statements:
\begin{itemize}
  \item it is easy to check that $\F_{\simp}$ is closed under extensions, submodules and quotients modules;
  \item and there is a subcategory $\T_{\simp}$ of $\mod\itLamb$ such that $(\T_{\simp},\F_{\simp})$ is a torsion pair
    (the corresponding torsion radical is denoted by $t_{\simp}$),
\end{itemize}
and have the following results.

\begin{proposition}[\!\!{\cite[Proposition 3.2]{zheng2020upper}}] \label{prop:rad-layer}
If $\simp=\varnothing$, then $\LL^{t_{\simp}}(M)=\ell\ell(M)$ for each $M\in\mod\itLamb$.
\end{proposition}

\begin{lemma}\label{lemm:basic-prop}
Let the pair $(\T_{\simp}, \F_{\simp})$ be a torsion pair as above.
\begin{itemize}
  \item[\rm(1)] {\rm(\!\cite[Proposition 5.9]{huard2009finitistic})}
    $\T_{\simp} = \{M\in\mod\itLamb \mid \top(M)\in\add(\simp')\}$;
  \item[\rm(2)] {\rm(\!\cite[Lemma 3.3(c)]{huard2009finitistic})}
    If $M\in\T_{\simp}$ and $M\ne0$, then $\LL^{t_{\simp}}(M)=\LL^{t_{\simp}}(\rad M)+1;$
  \item[\rm(3)] {\rm(\!\cite[Lemma 3.4(a)]{huard2009finitistic})}
    $\LL^{t_{\simp}}(\bigoplus_{i=1}^{n}M_{i})=\sup_{1\leqslant i\leqslant n}\LL^{t_{\simp}}(M_{i})$,
    where $M_{i}\in\mod\itLamb$ for all $1\leqslant i \leqslant n$.
\end{itemize}
\end{lemma}

\subsection{Short exact sequences and radical layer length}
In this section, we show two main results, i.e., Theorems \ref{thm:main 1} and \ref{thm:main 2}, of our paper.

\subsubsection{The first main result and some corollaries}
\label{subsubsect:main result 1}
In this subdivision, we show the following theorem which is the first result of our paper.

\begin{theorem}\label{thm:main 1}
Let $0  \longrightarrow L\longrightarrow  M\longrightarrow N\longrightarrow 0$
be an exact sequence in $\mod\itLamb$. Then
\[\max\{\LL^{t_{\simp}}(L), \LL^{t_{\simp}}(N)\}  \leqslant \LL^{t_{\simp}}(M) \leqslant \LL^{t_{\simp}}(L)+\LL^{t_{\simp}}(N).\]
In particular, if $\LL^{t_{\simp}}(L)=0$, then $\LL^{t_{\simp}}(N)=\LL^{t_{\simp}}(M);$
if $\LL^{t_{\simp}}(N)=0$, then $\LL^{t_{\simp}}(L)=\LL^{t_{\simp}}(M)$.
\end{theorem}

We need some lemmas for proving it.

\begin{lemma} \label{lemm:t-rad 1}
Let $X$ be a $\itLamb$-module.
\begin{itemize}
  \item[\rm(1)]{\rm(\!\cite[Proposition 5.9 (c)]{huard2013layer})}
    $t_{\simp}(\itLamb_{\itLamb})$ is a two side ideal of $\itLamb$ and $t_{\simp}(X) = X\cdot t_{\simp}(\itLamb_{\itLamb})$.
  \item[\rm(2)]{\rm(\!\cite[Propostion 3.5]{auslander1997representation})}
    $\rad X = X\cdot \rad(\itLamb_{\itLamb})$.
  \item[\rm(3)] $\rad(t_{\simp}(\itLamb_{\itLamb})) = t_{\simp}(\itLamb_{\itLamb})\cdot \rad(\itLamb_{\itLamb})$.
  \item[\rm(4)] $t_{\simp}(\rad(t_{\simp}(\itLamb_{\itLamb}))) = t_{\simp}(\itLamb_{\itLamb}) \cdot \rad(\itLamb_{\itLamb})\cdot t_{\simp}(\itLamb_{\itLamb})$.
  \item[\rm(5)] $t_{\simp}(F_{t_{\simp}}^{i}(\itLamb_{\itLamb}))$ is an ideal of $\itLamb$ for each $i \geqslant 0 $.
\end{itemize}
\end{lemma}

\begin{proof}
The statement (3) can be obtained by (2) in the case of $X = t_{\simp}(\itLamb_{\itLamb})$.
The statement (4) is obtained by (1) and (3), and the statement (5) is obtained by (1)--(4).
\end{proof}

\begin{lemma}\label{lemm:t-rad 2}
For any module $X\in\mod\itLamb$, we have
$t_{\simp} (F_{t_{\simp}}^{i}(X)) =
X\cdot t_{\simp}(F_{t_{\simp}}^{i}(\Lambda_{\Lambda}))$ for each $i \geqslant 0 $.
\end{lemma}

\begin{proof}
Lemma \ref{lemm:t-rad 1} (1) admits this statement under the case for $i=0$. Assume that if $i=n$ then $t_{\simp}(F_{t_{\simp}}^{n}(X)) = X\cdot t_{\simp}(F_{t_{\simp}}^{n}(\itLamb_{\itLamb}))$. For the case $i=n+1$, we have
\begin{align*}
    t_{\simp}(F_{t_{\simp}}^{n+1}(X))
& = t_{\simp}(F_{t_{\simp}}^{n}(F_{t_{\simp}}(X))) & \\
& = F_{t_{\simp}}(X)\cdot t_{\simp}(F_{t_{\simp}}^{n}(\itLamb_{\itLamb}))
 \tag{\text{using assumption}}\\
& = \rad(t_{\simp}(X))\cdot t_{\simp}(F_{t_{\simp}}^{n}(\itLamb_{\itLamb}))
 \tag{\text{$F_{t_{\simp}}=\rad \circled t_{\simp}$}} \\
& = t_{\simp}(X)\cdot \rad(\itLamb_{\itLamb})\cdot t_{\simp}(F_{t_{\simp}}^{n}(\itLamb_{\itLamb}))
 \tag{\text{Lemma \ref{lemm:t-rad 1} (2)}}  \\
& = t_{\simp}(X)\cdot \rad(t_{\simp}(F_{t_{\simp}}^{n}(\itLamb_{\itLamb})))
 \tag{\text{Lemma \ref{lemm:t-rad 1} (2)}}  \\
& = t_{\simp}(X)\cdot F_{t_{\simp}}^{n+1}(\itLamb_{\itLamb})
 \tag{\text{$F_{t_{\simp}} = \rad\circled t_{\simp}$}} \\
& = X\cdot  t_{\simp}(\itLamb_{\itLamb})\cdot F_{t_{\simp}}^{n+1}(\itLamb_{\itLamb})
 \tag{\text{Lemma \ref{lemm:t-rad 1} (1)}}  \\
& = X\cdot  t_{\simp }(F_{t_{\simp}}^{n+1}(\itLamb_{\itLamb})),
  \tag{\text{Lemma \ref{lemm:t-rad 1} (3) (4)}}
\end{align*}
which yields that this lemma holds by induction.
\end{proof}

\begin{lemma}\label{lemm:t-rad 3} \
\begin{itemize}
  \item[\rm(1)] The functor $t_{\simp}$ and $\rad $ preserve monomorphism and epimorphism.
  \item[\rm(2)] For each $i \geqslant 0$, $F_{t_{\simp}}^{i}$ and $t_{\simp}\circled F_{t_{\simp}}^{i}$ preserve monomorphism and epimorphism.
\end{itemize}
\end{lemma}

\begin{proof}
(1) The functor $t_{\simp}$ preserving monomorphism and epimorphism
is proven in \cite[Lemma 3.3]{zheng2020upper}.
By \cite[Lemma 3.6(a)]{huard2013layer}, $\rad$ preserves monomorphism,
and by \cite[Chapter V, Lemma 1.1]{assem2006elements}, $\rad$ preserves epimorphism.

(2) This statement is a direct corollary of (1).
\end{proof}

\begin{lemma}\label{lemm:t-rad 4}
Let $X$ be a $\itLamb$-module.
Then we have $t_{\simp}(F_{t_{\simp}}^{\LL^{t_{\simp}}(X)}(X)) = 0$.
\end{lemma}

\begin{proof}
By Definition \ref{radical-length}, we have this observation.
\end{proof}

Now, we prove the first main result of our paper.

\begin{proof}[{$\pmb{\text{The proof of Theorem \ref{thm:main 1}}}$}]
By Lemma \ref{lemm:t-rad 3}, we know that $F_{t_{\simp}}=\rad \circ\; t_{\simp}$  preserve monomorphism and epimorphism.
Thus, by \cite[Lemma 3.4(b)(c)]{huard2013layer}, we obtain
$\LL^{t_{\simp}}(L) \leqslant \LL^{t_{\simp}}(M)$ and
$\LL^{t_{\simp}}(N) \leqslant \LL^{t_{\simp}}(M)$, i.e.,
$\max\{\LL^{t_{\simp}}(L), \LL^{t_{\simp}}(N)\} \le \LL^{t_{\simp}}(M)$.

Next, we prove the second ``$\leqslant$'', it can be obtained by proving
\begin{align}\label{formula:rad-ses}
  t_{\simp}\big(F_{t_{\simp}}^{\LL^{t_{\simp}}(L)+\LL^{t_{\simp}}(N)}(M)\big)
\subseteq t_{\simp}\big(F_{t_{\simp}}^{\LL^{t_{\simp}}(L)}(L)\big)=0,
\end{align}
see Lemmas \ref{lemm:t-rad 3} (2) and \ref{lemm:t-rad 4}, and Definition \ref{radical-length}.

By Lemma \ref{lemm:t-rad 1} (5), $t_{\simp}(F_{t_{\simp}}^{i}(\itLamb_{\itLamb}))$ is an ideal of $\itLamb$ for each $i \ge 0$,
then we have
\begin{align*}
    \big( t_{\simp}(F_{t_{\simp}}^{\LL^{t_{\simp}}(N)}(M))+L \big)/L
& = \big(M\cdot  t_{\simp}(F_{t_{\simp}}^{\LL^{t_{\simp}}(N)}(\itLamb_{\itLamb}))+L\big)/L
\tag{Lemma \ref{lemm:t-rad 2}}
\\
& = M/L\cdot t_{\simp}(F_{t_{\simp}}^{\LL^{t_{\simp}}(N)}(\itLamb_{\itLamb}))
\\
& \cong N\cdot t_{\simp}(F_{t_{\simp}}^{\LL^{t_{\simp}}(N)}(\itLamb_{\itLamb}))
\tag{$M/L\cong N$}
\\
& = t_{\simp}(F_{t_{\simp}}^{\LL^{t_{\simp}}(N)}(N))
\tag{\text{Lemma \ref{lemm:t-rad 2}}}
\\
& = 0.
\tag{\text{Lemma \ref{lemm:t-rad 4}}}
\end{align*}
Thus, $t_{\simp}(F_{t_{\simp}}^{\LL^{t_{\simp}}(N)}(M))+L=L$.
It follows $t_{\simp}(F_{t_{\simp}}^{\LL^{t_{\simp}}(N)}(M)) \subseteq L$.
Furthermore, we obtain
\begin{align*}
  t_{\simp}\big(F_{t_{\simp}}^{\LL^{t_{\simp}}(L)+\LL^{t_{\simp}}(N)}(M)\big)
& = \big(t_{\simp} F_{t_{\simp}}^{\LL^{t_{\simp}}(L)}
    \circled F_{t_{\simp}}^{\LL^{t_{\simp}}(N)}\big)(M)
\\
& = \big(t_{\simp} F_{t_{\simp}}^{\LL^{t_{\simp}}(L)}
    \circled
    t_{\simp}F_{t_{\simp}}^{\LL^{t_{\simp}}(N)}\big)(M)
 \tag{$t_{\simp}^{2}=t_{\simp}$}
\\
& \subseteq t_{\simp}
      \big(
        F_{t_{\simp}}^{\LL^{t_{\simp}}(L)}
          (
            L
          )
      \big)
\\
& = 0, \tag{Lemma \ref{lemm:t-rad 4}}
\end{align*}
that is, (\ref{formula:rad-ses}) holds.
\end{proof}

Here are some corollaries about Theorem \ref{thm:main 1}.

\begin{corollary}\label{coro:LL}
Let $M\in [T]_{n+1}$.
Then $\LL^{t_{\simp}}(M)\leqslant (n+1) \LL^{t_{\simp}}(T).$
\end{corollary}

\begin{proof}
A module $M$ lying in $[T]_{n+1}$ provides a family of short exact sequences
\[ (0 \To{}{} Y_i \To{}{} Z_{i-1}\oplus Z_{i-1}' \To{}{} Z_i \To{}{} 0)_{1\leqslant i\leqslant n}, \]
where $Z_0=M$, $Y_1,\ldots, Y_n\in [T]_1$, $Z_i\in [T]_{n-i+1}$.
Then we obtain
\begin{align*}
  \LL^{t_{\simp}}(M) = \LL^{t_{\simp}}(Z_0)
& \leqslant \LL^{t_{\simp}}(Z_0 \oplus Z_0')
\\
& \leqslant \LL^{t_{\simp}}(Y_{1}) + \LL^{t_{\simp}}(Z_1)
  \tag{Lemma \ref{lemm:ext-pdim}}
\\
& \leqslant \LL^{t_{\simp}}(T)+\LL^{t_{\simp}}(Z_1\oplus Z_1')
\\
& \leqslant \LL^{t_{\simp}}(T)+\LL^{t_{\simp}}(Y_{2})+\LL^{t_{\simp}}(Z_{2})
\\
& \leqslant \LL^{t_{\simp}}(T)+\LL^{t_{\simp}}(T)+\LL^{t_{\simp}}(Z_{2})
 \tag{$Y_{2}\in\add T$}
\\
& = 2\LL^{t_{\simp}}(T)+\LL^{t_{\simp}}(Z_{2})
\\
& \ \ \ \ \vdots
\\
& \leqslant n\LL^{t_{\simp}}(T)+\LL^{t_{\simp}}(Z_{n})
\\
& \leqslant n\LL^{t_{\simp}}(T)+\LL^{t_{\simp}}(T)
 \tag{$Z_{n}\in\add T$}
\\
& = (n+1)\LL^{t_{\simp}}(T).
\end{align*}
\end{proof}

In the case for $\simp = \{ S \in \mod\itLamb \text{ is simple } \mid \pd S <\infty \}$,
we define $\LL^{\infty}(M) := \LL^{t_{\simp}}(M)$ in this situation,
see \cite[Definition 4.2]{huard2009finitistic}.
Then the following corollary holds.

\begin{corollary}\label{coro:LL-infty}
Let $0  \longrightarrow L\longrightarrow  M\longrightarrow N\longrightarrow 0$
be an exact sequence in $\mod\itLamb$.
Then the following statements hold.
\begin{itemize}
  \item[\rm(1)]
    $\max\{\LL(L), \LL(N)\}  \leqslant \LL(M) \leqslant \LL(L)+\LL(N).$
  \item[\rm(2)]
    $\max\{\LL^{\infty}(L), \LL^{\infty}(N)\} \leqslant \LL^{\infty}(M) \leqslant \LL^{\infty}(L) + \LL^{\infty}(N)$.
  \item[\rm(3)]
    If $\LL^{\infty}(L)=0$ {\rm(}resp. $\LL^{\infty}(N)=0${\rm)},
    then $\LL^{\infty}(N)=\LL^{\infty}(M)$ {\rm(}resp. $\LL^{\infty}(L)=\LL^{\infty}(M)${\rm)}.
\end{itemize}
\end{corollary}

\begin{proof}
(1) and (2) are particular cases of Theorem \ref{thm:main 1}. We need prove (3).

If $\LL^{\infty}(L)=0$, then we have
\[ \LL^{\infty}(N)
 = \max\{\LL^{\infty}(L), \LL^{\infty}(N)\}
 \leqslant \LL^{\infty}(M)
 \leqslant \LL^{\infty}(L) + \LL^{\infty}(N)
 = \LL^{\infty}(N) \]
by (2). Thus, $\LL^{\infty}(M)=\LL^{\infty}(N)$ in this situation.
We can show that $\LL^{\infty}(N)=0$ admits $\LL^{\infty}(L)=\LL^{\infty}(M)$ by similar way.
\end{proof}

\begin{remark} \rm
Note that the functions Loewy length $\LL$ and infinite layer length $\LL^{\infty}$
are particular radical layer length, see more details \cite{huard2013layer,huard2009finitistic}. Corollary \ref{coro:LL-infty} (1) is a classical result.
The first ``$\leqslant$'' in Corollary \ref{coro:LL-infty} (2)
is first established in \cite[Proposition 4.5 (a) (b)]{huard2009finitistic}.
\end{remark}

\begin{corollary}\label{coro:rad-LL}
Assume $\LL^{t_{\simp}}(\itLamb)=n$. Then for each $M\in \mod\itLamb$,
we have $\LL^{t_{\simp}}(F_{t_{\simp}}^{i}(M))\leqslant n-i$ for each $0\leqslant i \leqslant n$.
\end{corollary}

\begin{proof}
By \cite[Lemma 2.6]{zheng2020upper}, if $M\in \mod\itLamb$ is a $\itLamb$-module $M$ with $\LL^{t_{\simp}}(M)=n$, then $\LL^{t_{\simp}}(F_{t_{\simp}}^i(M))=n-i$ for each $0\leqslant i \leqslant n$.
Obviously, any module $M$ has a free precover $f: A^{\oplus I} \to M$ which induces a short exact sequence
\[ 0 \To{}{} \ker f \To{}{} A^{\oplus I} \To{}{} M \To{}{} 0, \]
where $I$ is a finite index set.
It yields $\LL^{t_{\simp}}(M) \leqslant \LL^{t_{\simp}}(\itLamb) \leqslant n$
by using Theorem \ref{thm:main 2} as required.
\end{proof}

\subsubsection{The second main result and some corollaries}
\label{subsubsect:main result 2}
In this subdivision, we show the following theorem which is the second result of our paper.

\begin{theorem}\label{thm:main 2}
If $\itLamb$ is representation-finite, then
\[\Big\{
         \Big\lceil
            \dfrac{\LL^{t_{\simp}}(\itLamb)}{d}
         \Big\rceil
         -1
     \ \Big|\
         d \in \NN
         \text{ and }
         1 \leqslant d < \LL^{t_{\simp}}(\itLamb)
 \Big\}
\subseteq \ospec(\itLamb),\]
where $\lceil \alpha\rceil$ is the least integer greater than $\alpha.$
\end{theorem}

We need the following lemma for proving Theorem \ref{thm:main 2}

\begin{lemma} \label{lemm:family of ses}
A family
\[(0 \To{}{} M_i \To{\varphi_i}{} M_{i-1} \To{}{} N_i \To{}{} 0)_{1\leqslant i\leqslant d}\]
of short exact sequences induces a short exact sequence
\[ 0 \To{}{} M_d \To{}{} M_0 \To{}{} N_{[0,d]} \To{}{} 0 \]
satisfying $\LL^{t_{\simp}}(N_{[0,d]}) \leqslant\sum_{i=1}^{d}\LL^{t_{\simp}}(N_{i})$,
where $d\geqslant 1.$
\end{lemma}

\begin{proof}
The case for $d=1$ is trivial. Next, we show this lemma holds for $d\geqslant 2$.

For simplification, we define $N_{[u,u+i]} := M_u/M_{u+i}$ which is induced by the cokernel of the composition
\[\varphi_{[u,u+i]}:=\varphi_{u+1}\circled\cdots\circled\varphi_{u+i}: M_{u+i} \to M_u\]
\begin{center}
($0\leqslant u<u+i\leqslant d$, $1\leqslant i\leqslant d-u$).
\end{center}
Since composition $\varphi_{[u-1,u+i]}$ is always monomorphic,
$\varphi_{[0,d]}$ induces a short exact sequence which is of the form
\begin{center}
$0 \To{}{} M_d \To{\varphi_{[0,d]}}{} M_0 \To{}{} N_{[0,d]} \To{}{} 0$,
\end{center}
where $N_{[0,d]}\cong M_0/M_d$.
Moreover, for any $1\leqslant i\leqslant d$, the epimorphism $\phi_{[d-(i+1),d]}: M_{d-(i+1)} \to N_{[d-(i+1), d]}$ induced by the cokernel of $\varphi_{[d-(i+1),d]}$
and the monomorphism $\phi_{d-i}: N_{[d-i,d]} \to N_{[d-(i+1),d]}$ obtained by the universal property of the cokernel of $\varphi_{[d-i,d]}$ provide the following diagram
\begin{center}
\begin{tikzpicture}
\filldraw [top color = cyan!25, bottom color=blue!25]
      (8.6,6.5) to[out=   0,in=  90] (9.8,3.3) to[out= -90,in=   0]
      (8.6,0.0) to[out= 180,in= -90] (7.4,3.3) to[out=  90,in= 180]
      (8.6,6.5);
\draw(7,4) node{$\spadesuit$};
\draw(0,0) node[above right]{$
\xymatrix@C=1.5cm{
 & & 0 \ar[d] & 0 \ar[d]
 & \\
   0 \ar[r]
 & M_{d}
   \ar[r]^{\varphi_{[d-i,d]}}
   \ar@{=}[d]
 & M_{d-i}
   \ar[r]
   \ar[d]^{\varphi_{d-i}}
 & N_{[d-i,d]}
   \ar[r]
   \ar[d]^{\phi_{d-i}}
 & 0
 \\
   0 \ar[r]
 & M_{d}
   \ar[r]_{\varphi_{[d-(i+1),d]}}
 & M_{d-(i+1)}
   \ar[r]_{\phi_{[d-(i+1),d]}}
   \ar[d]
 & N_{[d-(i+1),d]}
   \ar[d]
   \ar[r]
 & 0 \\
 &
 & N_{d-(i+1)}
   \ar@{=}[r]
   \ar[d]
 & N_{d-{i+1}}
   \ar[d]
 &  \\
 && 0 & 0 & }
$};
\end{tikzpicture}
\end{center}
commutes, whose rows and columns are exact, and the commutative square marked by ``$\spadesuit$''
is both a push-out and pull-back square. Of course, the above diagram can also be induced by using the snake lemma.
Then we obtain a family
\[ (0 \To{}{} N_{[d-j,d]} \To{}{} N_{[d-(j+1),d]} \To{}{} N_{d-(j+1)} \To{}{} 0)_{0\leqslant j\leqslant d-2} \]
of short exact sequences (see the shadow of the above diagram), where $N_{[d,d]} := N_d$.
Then, by Theorem \ref{thm:main 2}, we have
\begin{align*}
    \LL^{t_{\simp}}(N_{[1,d]})
& \leqslant \LL^{t_{\simp}}(N_{[2,d]})
   +  \LL^{t_{\simp}}(N_{1})
  \\
& \leqslant \LL^{t_{\simp}}(N_{[3,d]})
   +  \LL^{t_{\simp}}(N_{2})
   +  \LL^{t_{\simp}}(N_{1})\\
& \,\,\,\,\vdots \\
& \leqslant \LL^{t_{\simp}}(N_d)
   +  \cdots
   +  \LL^{t_{\simp}}(N_2)
   +  \LL^{t_{\simp}}(N_1)
   =  \sum_{i=1}^{d} \LL^{t_{\simp}}(N_{i}).
\end{align*}
\end{proof}

Now we prove Theorem \ref{thm:main 2}.

\begin{proof}[$\pmb{\text{The proof of Theorem \ref{thm:main 2}}}$]
This proof is divided to two parts (i) and (ii).

(i) In the case of $\LL^{t_\simp}(\itLamb)=1$, we have
\begin{align*}
\Big\{
         \Big\lceil
            \dfrac{\LL^{t_{\simp}}(\itLamb)}{d}
         \Big\rceil
         -1
     \ \Big|\
         d \in \NN
         \text{ and }
         1 \leqslant d < \LL^{t_{\simp}}(\itLamb)
 \Big\}\Big|_{\LL^{t_{\simp}}(\itLamb)=1}
= \{ 0 \} \subseteq \ospec(\itLamb).
\end{align*}

(ii) Assume that $\LL^{t_{\simp}}(\itLamb) = n\geqslant 2$, and for any $1\leqslant d < \LL^{t_{\simp}}(\itLamb)$, assume $n=dm+r$, where $0\leqslant r < d$.
Thus, we need show $m-1\in\ospec(\itLamb)$ for proving this theorem.
To do this, we construct a $\itLamb$-module $T$ such that $\gent_{\mod\itLamb}(T)=m-1$ in this proof.

First of all, notice that, for arbitrary $M\in\mod\itLamb$, we have the following two short exact sequences
\begin{align}\label{formula:twoses}
& 0 \To{}{} t_{\simp}\circled F^{i-1}_{t_{\simp}}(M)
    \To{f_{i-1}}{} F^{i-1}_{t_{\simp}}(M)
    \To{}{} q_{t_{\simp}}\circled F^{i-1}_{t_{\simp}}(M) \nonumber
    \To{}{} 0 \\ \text{ and }
& 0 \To{}{} F^{i}_{t_{\simp}}(M)
    \To{}{} t_{\simp}\circled F^{i-1}_{t_{\simp}}(M)
    \To{g_{i-1}}{} \top(t_{\simp}\circled F^{i}_{t_{\simp}}(M) )
    \To{}{} 0
\end{align}
for each $1\leqslant i\leqslant n-1$. Here, $F^{i}_{t_{\simp}}(M) = \rad (t_{\simp}\circled F^{i-1}_{t_{\simp}}(M))$, see Definition \ref{radical-length}.
It follows that the following push-out square
\[\xymatrix{
 & t_{\simp}\circled F_{t_{\simp}}^{i-1}(M)
   \ar[r]^{f_{i-1}} \ar[d]_{g_{i-1}}
 & F_{t_{\simp}}^{i-1}(M)
   \ar[d] \\
 & \top (t_{\simp}\circled F_{t_{\simp}}^{i-1}(M))
   \ar[r]
 & D_i
}\]
provides the diagram
\[\xymatrix{
& 0 \ar[d] & 0 \ar[d] & &  \\ 
& F_{t_{\simp}}^{i}(M)
  \ar@{=}[r] \ar[d]
& F_{t_{\simp}}^{i}(M)
  \ar[d]
&
&  \\ 
  0 \ar[r]
& t_{\simp}\circled F_{t_{\simp}}^{i-1}(M)
  \ar[r]^{f_{i-1}}\ar[d]_{g_{i-1}}
& F_{t_{\simp}}^{i-1}(M)
  \ar[r]
  \ar[d]
& q_{t_{\simp}}\circled F_{t_{\simp}}^{i-1}(M)
  \ar@{=}[d]
  \ar[r]
& 0  \\ 
  0 \ar[r]
& \top(t_{\simp}\circled F_{t_{\simp}}^{i-1}(M))
  \ar[r]
  \ar[d]
& D_{i}
  \ar[r]
  \ar[d]
& q_{t_{\simp}}\circled F_{t_{\simp}}^{i-1}(M)
  \ar[r]
& 0 \\
& 0 & 0 & & }
\]
commutes, whose rows and columns are exact. Thus, we obtain that
\begin{align}\label{formula:1st}
\LL^{t_{\simp}}(D_i)
\leqslant \LL^{t_{\simp}}(\top(t_{\simp}\circled F_{t_{\simp}}^{i-1}(M)))
  + \LL^{t_{\simp}}(q_{t_{\simp}}\circled F_{t_{\simp}}^{i-1}(M))
\leqslant 1+0 = 1
\end{align}
holds for each $1\leqslant i \leqslant n-1$, and there is a family
\[ (0 \To{}{} F^i_{t_{\simp}}(M)
      \To{\varphi_i}{} F^{i-1}_{t_{\simp}}(M)
      \To{}{} D_i
      \To{}{} 0
   )_{1\leqslant i\leqslant n-1} \]
of short exact sequences.

Second, the composition $\varphi_{[ud,(u+1)d]} := \varphi_{(u+1)d} \circled \cdots \circled \varphi_{ud+1}$
($1\leqslant u < m$) (resp. $\varphi_{[md,md+r]} := \varphi_{md+r} \circled \cdots \circled \varphi_{md+1}$)
is monomorphic since all $\varphi_i$ are monomorphic,
it induces a short exact sequence which is of the form
\begin{align}
& 0 \To{}{} F^{(u+1)d}_{t_{\simp}}(M)
    \To{}{} F^{ud}_{t_{\simp}}(M)
    \To{}{} D_{[ud,(u+1)d]}
    \To{}{} 0  \label{formula:2nd-1} \\
(\text{resp. }
& 0 \To{}{} F^{(m-1)d-r}_{t_{\simp}}(M)
    \To{}{} F^{(m-2)d}_{t_{\simp}}(M)
    \To{}{} D_{[(m-2)d,(m-1)d-r]}
    \To{}{} 0). \label{formula:2nd-2}
\end{align}
where
\begin{center}
$D_{[ud,(u+1)d]} \cong F^{ud}_{t_{\simp}}(M)/F^{(u+1)d}_{t_{\simp}}(M)$

(resp. $D_{[(m-2)d,(m-1)d-r]} \cong F^{(m-2)d}_{t_{\simp}}(M)/F^{(m-1)d-r}_{t_{\simp}}(M)$).
\end{center}
By Lemma \ref{lemm:family of ses} and (\ref{formula:1st}), we have
\[ \LL^{t_{\simp}}(D_{[ud,(u+1)d]})
\leqslant \sum_{i=ud+1}^{(u+1)d} \LL^{t_{\simp}}(D_i)
= d \]
\[\Big(\text{resp. }
\LL^{t_{\simp}}(D_{[(m-2)d,(m-1)d-r]})
\leqslant \sum_{i = (m-2)d+1}^{(m-1)d-r} \LL^{t_{\simp}}(D_i)
\leqslant d
\Big). \]
Then, by using (\ref{formula:2nd-1}) and (\ref{formula:2nd-2}), we have
\begin{align} \label{formula:2nd-3}
[M]_1
& = [\rad^0\circled t_{\simp}^0 (M)]_1 =[F_{t_{\simp}}^{0}(M)]_{1}
  \nonumber \\
& \subseteq [F_{t_{\simp}}^{d}(M)]_1
  \multi [D_{[0,d]}]_{1}
  \nonumber \\
& \subseteq [F_{t_{\simp}}^{2d}(M)]_1
  \multi [D_{[d,2d]}]_{1}\multi [D_{[0,d]}]_1
  \nonumber \\
& \;\;\; \vdots
  \nonumber \\
& \subseteq [F_{t_{\simp}}^{(m-2)d}(M)]_1
  \multi [D_{[(m-3)d,(m-2)d]}]_1
  \multi \cdots
  \multi [D_{[0,d]}]_1
  \nonumber \\
& \subseteq [F_{t_{\simp}}^{(m-1)d-r}(M)]_1
  \multi [D_{[(m-2)d, (m-1)d-r]}]_1
  \multi [D_{[(m-3)d, (m-2)d]}]_1
  \multi \cdots
  \multi [D_{[0,d]}]_1.
\end{align}

Finally, we construct the $\itLamb$-module $T$ as follows
\[ T=\bigoplus_{Z\in \mathcal{W}_d} Z, \]
where $\mathcal{W}_d := \{Z \in \mathsf{ind}(\mod\itLamb) \mid \LL^{t_{\simp}}(Z) \leqslant d\}$.
We claim that $[T]_m=\mod\itLamb$ and $[T]_{m-1}\subsetneq\mod\itLamb$ hold.
Indeed, by Lemma \ref{lemm:basic-prop} (3), we have
\[\LL^{t_{\simp}}(T)
= \LL^{t_{\simp}}
  \bigg(
    \bigoplus_{Z\in \mathcal{W}_d} Z
  \bigg)
= \sup\{\LL^{t_{\simp}}(Z) \mid Z\in\mathcal{W}_d \}
\leqslant d.\]
Then we have $[M]_1 \subseteq \overbrace{[T]_1 \multi [T]_1 \multi \cdots \multi [T]_1}^{m} = [T]_m$ by (\ref{formula:2nd-3}). In this case, $\mod\itLamb = [T]_{m}$.

On the other hand, if $[T]_{m-1} = \mod\itLamb$,
then, by Corollary \ref{coro:LL}, we obtain
\[  \LL^{t_{\simp}}(\itLamb)
\leqslant (m-1)\LL^{t_{\simp}}(T)
\leqslant (m-1)d=md-d< md-j
 =  \LL^{t_{\simp}}(\itLamb), \]
This is a contradiction.

Therefore, $\gent_{\mod\itLamb}(T)=m-1$, and then $m-1\in \ospec(\itLamb)$ holds.
\end{proof}

\begin{example}\rm
Take \[\itLamb=\kk\Q/\langle (\alpha_1\alpha_2\alpha_3\alpha_4)^5 \rangle,\]
where $\Q$ is defined as
\[\xymatrix{
1\ar[r]^{\alpha_{1}} & 2 \ar[d]^{\alpha_{2}}\\
4\ar[u]^{\alpha_{4}} & 3.\ar[l]^{\alpha_{3}}
}\]
We have seven items as follows:
\begin{itemize}
\item[(a)]
  Let $\simp=\varnothing$, then $\LL^{t_{\simp}}(\itLamb)$ by Proposition \ref{prop:rad-layer}.
  On the other hand,
  $\LL(\itLamb)=\sup\{\LL(P(i)) \mid 1\leqslant i \leqslant 4\}=23$,
  then by Theorem \ref{thm:main 2}, we have
  \[\{0,1,2,3,4,5,7,11,22\}\subseteq \ospec(\itLamb). \]

\item[(b)] Let $\simp=\{S(1)\},$ then $\LL^{t_{\simp}}(\itLamb) = \sup\{\LL^{t_{\simp}}(P(i)) \mid 1\leqslant i \leqslant 4\}=\LL^{t_{\simp}}(P(2))=18$.
  Furthermore, by Theorem \ref{thm:main 2}, we have
  \[\{0,1,2,3,4,5,8,17\}\subseteq \ospec(\itLamb).\]

\item[(c)] Let $\simp=\{S(2)\}, \{S(3)\},$ or $\{S(4)\}$, then
  $ \LL^{t_{\simp}}(\itLamb)=\sup\{\LL^{t_{\simp}}(P(i)) \mid 1\leqslant i \leqslant 4\}
   = \LL^{t_{\simp}}(P(2))=17.$
  Furthermore, by Theorem \ref{thm:main 2}, we have
  \[\{0,1,2,3,4,5,8,16\}\subseteq \ospec(\itLamb).\]

\item[(d)]  Let $\simp=\{S(1),S(2)\}, \{S(1),S(3)\},$ or $\{S(1),S(4)\}$, then
  $\LL^{t_{\simp}}(\itLamb) = \sup\{\LL^{t_{\simp}}(P(i)) \mid 1\leqslant i \leqslant 4\}=\LL^{t_{\simp}}(P(2))=12.$
  Furthermore, by Theorem \ref{thm:main 2}, we have
  \[\{0,1,2,3,5,11\}\subseteq \ospec(\itLamb).\]

\item[(e)]  Let $\simp=\{S(2),S(3)\}, \{S(2),S(4)\},$ or $\{S(3),S(4)\}$, then
  $\LL^{t_{\simp}}(\itLamb) = \sup\{\LL^{t_{\simp}}(P(i)) \mid 1 \leqslant i \leqslant 4\}=\LL^{t_{\simp}}(P(2))=11$.
  Furthermore, by Theorem \ref{thm:main 2}, we have
  \[\{0,1,2,3,5,10\}\subseteq \ospec(\itLamb).\]

\item[(f)]  Let $\simp=\{S(1),S(2),S(3)\},\{S(1),S(2),S(4)\},$ or $\{S(1),S(3),S(4)\},$
  then $\LL^{t_{\simp}}(\itLamb) = \sup\{\LL^{t_{\simp}}(P(i)) \mid 1\leqslant i \leqslant 4\} = \LL^{t_{\simp}}(P(2))=6$.
  Furthermore, by Theorem \ref{thm:main 2}, we have
  \[\{0,1,2,3,5\}\subseteq \ospec(\itLamb). \]

\item[(g)]  Let $\simp=\{S(2),S(3),S(4)\}$, then $\LL^{t_{\simp}}(\itLamb) = \sup\{\LL^{t_{\simp}}(P(i)) \mid 1\leqslant i \leqslant 4\}=\LL^{t_{\simp}}(P(2))=5.$
  Furthermore, by Theorem \ref{thm:main 2}, we have
  \[\{0,1,2,4\}\subseteq \ospec(\itLamb).\]
\end{itemize}

By (a)--(g), we have
\[\{0,1,2,3,4,5,6,7,8,10,11,16,17,22\}\subseteq \ospec(\itLamb).\]
\end{example}

\begin{prop-coro}\label{prop-coro:oriented cycle}
The Orlov spectrum $\ospec(\itLamb)$ of the finite-dimensional algebra $\itLamb=\kk\Q/I$ given by
the quiver shown in Figure \ref{fig:path alg3} and the admissible ideal
$I=\langle \alpha_{1}\alpha_{2}\cdots\alpha_{m}\rangle$ and $2\leqslant m \leqslant n-1$
satisfies
\begin{align}\label{formula:oriented cycle}
  \{0,1,2,\cdots,m+n-1\}\subseteq\ospec(\itLamb).
\end{align}
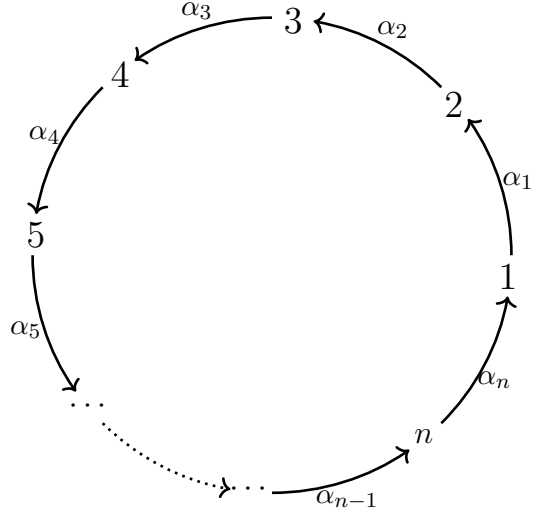
\begin{figure}[H]
\centering
\begin{tikzpicture}[scale=0.9]
\draw[->][line width=1pt][rotate=  0] (3.5,0) arc (0:35:3.5);
\draw[->][line width=1pt][rotate= 45] (3.5,0) arc (0:35:3.5);
\draw[->][line width=1pt][rotate= 90] (3.5,0) arc (0:35:3.5);
\draw[->][line width=1pt][rotate=135] (3.5,0) arc (0:35:3.5);
\draw[->][line width=1pt][rotate=180] (3.5,0) arc (0:35:3.5);
\draw[->][line width=1pt][rotate=225] (3.5,0) arc (0:35:3.5)[dotted];
\draw[->][line width=1pt][rotate=270] (3.5,0) arc (0:35:3.5);
\draw[->][line width=1pt][rotate=315] (3.5,0) arc (0:35:3.5);
\draw[rotate=  0] (3.45,-0.31) node{\Large\rotatebox{0}{$1$}};
\draw[rotate= 45] (3.45,-0.31) node{\Large\rotatebox{0}{$2$}};
\draw[rotate= 90] (3.45,-0.31) node{\Large\rotatebox{0}{$3$}};
\draw[rotate=135] (3.45,-0.31) node{\Large\rotatebox{0}{$4$}};
\draw[rotate=180] (3.45,-0.31) node{\Large\rotatebox{0}{$5$}};
\draw[rotate=225] (3.45,-0.31) node{\large\rotatebox{0}{$\cdots$}};
\draw[rotate=270] (3.45,-0.31) node{\large\rotatebox{0}{$\cdots$}};
\draw[rotate=315] (3.45,-0.31) node{\large\rotatebox{0}{$n$}};
\draw[rotate=  0] (3.6,1.1) node{\rotatebox{0}{$\alpha_{1}$}};
\draw[rotate= 45] (3.6,1.1) node{\rotatebox{0}{$\alpha_{2}$}};
\draw[rotate= 90] (3.6,1.1) node{\rotatebox{ 0}{$\alpha_{3}$}};
\draw[rotate=135] (3.6,1.1) node{\rotatebox{0}{$\alpha_{4}$}};
\draw[rotate=180] (3.6,1.1) node{\rotatebox{0}{$\alpha_{5}$}};
\draw[rotate=270] (3.6,1.1) node{\rotatebox{0}{$\alpha_{n-1}$}};

\draw[rotate=315] (3.6,1.0) node{\rotatebox{0}{$\alpha_{n}$}};
\end{tikzpicture}
\caption{The path algebra in Corollary/Proposition \ref{prop-coro:oriented cycle}}
\label{fig:path alg3}
\end{figure}
\end{prop-coro}

\begin{proof}
We have the five items as follows.
\begin{itemize}
\item[(a)] If $\simp = \varnothing$, then $\LL^{t_{\simp}}(\itLamb) = \LL(\itLamb)m+n$.
\item[(b)] If $\simp = \{S(2),S(3),\ldots,S(i)\}$ (for any $i$ satisfying $2\leqslant i\leqslant m$), then
  $\LL^{t_{\simp}}(\itLamb) = \sup\{P(i) \mid 1\leqslant i\leqslant n\} = \LL^{t_{\simp}}(P(2)) = m+n+2-2i$.
\item[(c)] If $\simp = \{S(1),S(2),S(3),\ldots,S(i)\}$ (for any $i$ satisfying $2\leqslant i\leqslant m$), then
  $\LL^{t_{\simp}}(\itLamb) = \sup\{P(i) \mid 1\leqslant i\leqslant n\} = \LL^{t_{\simp}}(P(2)) = m+n+1-2i$.
\item[(d)] If $\simp=\{S(1),S(2),S(3),\ldots,S(m),S(m+1),\ldots,S(i)\}$ (for any $i$ satisfying $m+2\leqslant i\leqslant n$),
  then $\LL^{t_{\simp}}(\itLamb)=\sup\{P(i) \mid 1\leqslant i \leqslant n\} = \LL^{t_{\simp}}(P(2))=n+2-i$.
\item[(e)] If $\simp=\{S(2),S(3),\ldots,S(m),S(m+1),\ldots,S(n)\}$,
  then $\LL^{t_{\simp}}(\itLamb)=\sup\{P(i) \mid 1\leqslant i \leqslant n\} = \LL^{t_{\simp}}(P(2))=1$.

\end{itemize}
By Theorem \ref{thm:main 2}, we obtain $\LL^{t_{\simp}}(\itLamb)-1\in\ospec(\itLamb)$.
Thus, (a)--(e) yields that $\{1,2,\cdots,m+n-1\}\subseteq\ospec(\itLamb)$.
Moreover, $\itLamb$ is representation-finite, it follows $0\in \ospec(\itLamb)$,
see Lemma \ref{lemm:0-in-ospec}. Therefore, (\ref{formula:oriented cycle}) holds.
\end{proof}

\section{Application: The Orlov spectrum with type $\protect\A_{n}$}

Let $\dA_n$ be the linearly oriented quiver with type $\A_n$, i.e.,
\[\dA_n = \xymatrix{1\ar[r]^{a_1} & 2 \ar[r]^{a_2} & \cdots \ar[r]^{a_{n-1}} & n}. \]
In this subsection, we compute the Orlov spectrum $\ospec(A_n)$ of $A_n:=\kk\dA_n$.

\begin{proposition} \label{prop-An}
$\{0,1,2,\cdots,n-1\}\subseteq\ospec(A_n)$.
\end{proposition}

\begin{proof}

For each $1 \leqslant i\leqslant n-1$, we have the following three items.
\begin{itemize}
  \item[(a)] For the case of $\simp=\varnothing$, we have $\LL(A_n)-1=n-1\in\ospec(A_n)$.
  \item[(b)] For the case of $\simp=\{S(1)\}$, we have $\LL^{t_{\simp}}(A_n) = \sup \{P(i) \mid 1\leqslant i\leqslant n\} = n-1$.
      Then we obtain $\LL^{t_{\simp}}(A_n) -1 = n-2 \in \ospec(A_n)$ by Theorem \ref{thm:main 2}.
  \item[(c)] For the case of $\simp=\{S(1),\ldots,S(j)\}$ ($1\leqslant j\leqslant n-1$), we have $\LL^{t_{\simp}}(A_n) = \sup\{P(i) \mid 1\leqslant i\leqslant n\}=n-j$.
      Then we obtain $\LL^{t_{\simp}}(A_n)-1=(n-j)-1 = n-j-1 \in \ospec(A_n)$.
\end{itemize}
The above three items yield $\{0,1,2,\cdots,n-1\} \subseteq \ospec(A_n)$.
\end{proof}

Next, we define some notation to describe the Auslander-Reiten quiver $\Gamma(A_n)$ of $A_n$.
Recall that a path $\wp_{[i,j]}=a_ia_{i+1}\cdots a_j$ ($j\geqslant i$) on $\dA_n$ corresponds to an indecomposable representation $M_{[i,j]} = ((M_{[i,j]})_k, (M_{[i,j]})_{\alpha})_{k\in(\dA_n)_0, \alpha\in(\dA_n)_1}$ satisfying
\[ (M_{[i,j]})_k =
  \begin{cases}
    \kk, & \text{ if } i\leqslant k \leqslant j+1; \\
    0,   & \text{ otherwise},
  \end{cases}
\text{\ and\ }
(M_{[i,j]})_{\alpha} =
  \begin{cases}
    1_{\kk}, & \text{ if } \alpha\in\{a_i,a_{i+1},\cdots,a_{j}\}; \\
    0, & \text{ otherwise}.
  \end{cases}
\]
Then each irreducible morphisms in $\mod A_n$ is one of the following form:
\begin{itemize}
  \item the irreducible monomorphism $f^{+}_{[i,j]}: M_{[i,j]} \to M_{[i-1,j]}$ ($2\leqslant i\leqslant n$),
  \item the irreducible epimorphism $f^{-}_{[i,j]}: M_{[i,j]} \to M_{[i,j-1]}$ ($2\leqslant j\leqslant n$);
\end{itemize}
and in this case, $\Gamma(A_n)=(\Gamma(A_n)_0, \Gamma(A_n)_1, \s, \t)$, where
\begin{itemize}
  \item $\Gamma(A_n)_0 = \{M_{[i,j]} \mid 1\leqslant i\leqslant j\leqslant n\}$;
  \item $\Gamma(A_n)_1 = \{ f^{+}_{[i,j]} \mid 2\leqslant i\leqslant n, 1\leqslant j\leqslant n\} \cup \{f^{-}_{[i,j]} \mid 1\leqslant i\leqslant n, 2\leqslant j\leqslant n\}$;
  \item $\s$ and $\t$ respectively send any arrow $f^{\pm}_{[-,-]}$ to its starting and ending points, see \Pic\ \ref{fig:linaer-An}
\end{itemize}

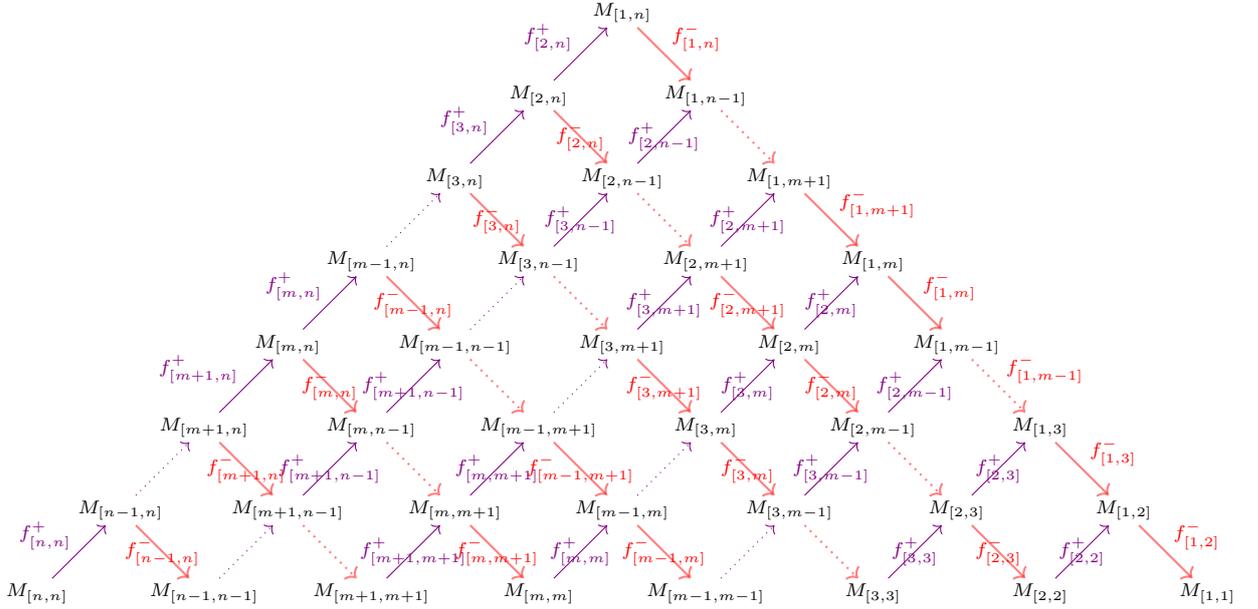
\begin{figure}[htbp]
\tiny
\hspace{-0.8cm}
\begin{tikzpicture}[scale=1.1]
\draw (-2+0,-1+0) node{$M_{[n  , n  ]}$}
      (-2+1,-1+1) node{$M_{[n-1, n  ]}$}
      (-2+2,-1+2) node{$M_{[m+1, n  ]}$}
      (-2+3,-1+3) node{$M_{[m  , n  ]}$}
      (-2+4,-1+4) node{$M_{[m-1, n  ]}$}
      (-2+5,-1+5) node{$M_{[3  , n  ]}$}
      (-2+6,-1+6) node{$M_{[2  , n  ]}$}
      (-2+7,-1+7) node{$M_{[1  , n  ]}$};
\draw ( 0+0,-1+0) node{$M_{[n-1, n-1]}$}
      ( 0+1,-1+1) node{$M_{[m+1, n-1]}$}
      ( 0+2,-1+2) node{$M_{[m  , n-1]}$}
      ( 0+3,-1+3) node{$M_{[m-1, n-1]}$}
      ( 0+4,-1+4) node{$M_{[3  , n-1]}$}
      ( 0+5,-1+5) node{$M_{[2  , n-1]}$}
      ( 0+6,-1+6) node{$M_{[1  , n-1]}$};
\draw ( 2+0,-1+0) node{$M_{[m+1, m+1]}$}
      ( 2+1,-1+1) node{$M_{[m  , m+1]}$}
      ( 2+2,-1+2) node{$M_{[m-1, m+1]}$}
      ( 2+3,-1+3) node{$M_{[3  , m+1]}$}
      ( 2+4,-1+4) node{$M_{[2  , m+1]}$}
      ( 2+5,-1+5) node{$M_{[1  , m+1]}$};
\draw ( 4+0,-1+0) node{$M_{[m  , m  ]}$}
      ( 4+1,-1+1) node{$M_{[m-1, m  ]}$}
      ( 4+2,-1+2) node{$M_{[3  , m  ]}$}
      ( 4+3,-1+3) node{$M_{[2  , m  ]}$}
      ( 4+4,-1+4) node{$M_{[1  , m  ]}$};
\draw ( 6+0,-1+0) node{$M_{[m-1, m-1]}$}
      ( 6+1,-1+1) node{$M_{[3  , m-1]}$}
      ( 6+2,-1+2) node{$M_{[2  , m-1]}$}
      ( 6+3,-1+3) node{$M_{[1  , m-1]}$};
\draw ( 8+0,-1+0) node{$M_{[3  , 3  ]}$}
      ( 8+1,-1+1) node{$M_{[2  , 3  ]}$}
      ( 8+2,-1+2) node{$M_{[1  , 3  ]}$};
\draw (10+0,-1+0) node{$M_{[2  , 2  ]}$}
      (10+1,-1+1) node{$M_{[1  , 2  ]}$};
\draw (12+0,-1+0) node{$M_{[1  , 1  ]}$};
\draw[violet] (-1.5+0,-0.3+0) node[left]{$f^+_{[n  , n  ]}$}
              (-1.5+2,-0.3+2) node[left]{$f^+_{[m+1, n  ]}$}
              (-1.5+3,-0.3+3) node[left]{$f^+_{[m  , n  ]}$}
              (-1.5+5,-0.3+5) node[left]{$f^+_{[3  , n  ]}$}
              (-1.5+6,-0.3+6) node[left]{$f^+_{[2  , n  ]}$};
\draw[violet] ( 0.5+1,-0.5+1) node{$f^+_{[m+1, n-1]}$}
              ( 0.5+2,-0.5+2) node{$f^+_{[m+1, n-1]}$}
              ( 0.5+4,-0.5+4) node{$f^+_{[3  , n-1]}$}
              ( 0.5+5,-0.5+5) node{$f^+_{[2  , n-1]}$};
\draw[violet] ( 2.5+0,-0.5+0) node{$f^+_{[m+1, m+1]}$}
              ( 2.5+1,-0.5+1) node{$f^+_{[m  , m+1]}$}
              ( 2.5+3,-0.5+3) node{$f^+_{[3  , m+1]}$}
              ( 2.5+4,-0.5+4) node{$f^+_{[2  , m+1]}$};
\draw[violet] ( 4.5+0,-0.5+0) node{$f^+_{[m  , m  ]}$}
              ( 4.5+2,-0.5+2) node{$f^+_{[3  , m  ]}$}
              ( 4.5+3,-0.5+3) node{$f^+_{[2  , m  ]}$};
\draw[violet] ( 6.5+1,-0.5+1) node{$f^+_{[3  , m-1]}$}
              ( 6.5+2,-0.5+2) node{$f^+_{[2  , m-1]}$};
\draw[violet] ( 8.5+0,-0.5+0) node{$f^+_{[3  , 3  ]}$}
              ( 8.5+1,-0.5+1) node{$f^+_{[2  , 3  ]}$};
\draw[violet] (10.5+0,-0.5+0) node{$f^+_{[2  , 2  ]}$};
\draw[   red] (11.5-0,-0.3+0) node[right]{$f^-_{[1  , 2  ]}$}
              (11.5-1,-0.3+1) node[right]{$f^-_{[1  , 3  ]}$}
              (11.5-2,-0.3+2) node[right]{$f^-_{[1  , m-1]}$}
              (11.5-3,-0.3+3) node[right]{$f^-_{[1  , m  ]}$}
              (11.5-4,-0.3+4) node[right]{$f^-_{[1  , m+1]}$}
              (11.5-6,-0.3+6) node[right]{$f^-_{[1  , n  ]}$};
\draw[   red] ( 9.5-0,-0.5+0) node{$f^-_{[2  , 3  ]}$}
              ( 9.5-2,-0.5+2) node{$f^-_{[2  , m  ]}$}
              ( 9.5-3,-0.5+3) node{$f^-_{[2  , m+1]}$}
              ( 9.5-5,-0.5+5) node{$f^-_{[2  , n  ]}$};
\draw[   red] ( 7.5-1,-0.5+1) node{$f^-_{[3  , m  ]}$}
              ( 7.5-2,-0.5+2) node{$f^-_{[3  , m+1]}$}
              ( 7.5-4,-0.5+4) node{$f^-_{[3  , n  ]}$};
\draw[   red] ( 5.5-0,-0.5+0) node{$f^-_{[m-1, m  ]}$}
              ( 5.5-1,-0.5+1) node{$f^-_{[m-1, m+1]}$}
              ( 5.5-3,-0.5+3) node{$f^-_{[m-1, n  ]}$};
\draw[   red] ( 3.5-0,-0.5+0) node{$f^-_{[m  , m+1]}$}
              ( 3.5-2,-0.5+2) node{$f^-_{[m  , n  ]}$};
\draw[   red] ( 1.5-1,-0.5+1) node{$f^-_{[m+1, n  ]}$};
\draw[   red] (-0.5-0,-0.5+0) node{$f^-_{[n-1, n  ]}$};
%
\draw[shift={(-1,-1)}][violet][->] (-1+0.18, 0+0.18)--( 0-0.18, 1-0.18);
\draw[shift={( 0, 0)}][violet][->] (-1+0.18, 0+0.18)--( 0-0.18, 1-0.18)[dotted];
\draw[shift={( 1, 1)}][violet][->] (-1+0.18, 0+0.18)--( 0-0.18, 1-0.18);
\draw[shift={( 2, 2)}][violet][->] (-1+0.18, 0+0.18)--( 0-0.18, 1-0.18);
\draw[shift={( 3, 3)}][violet][->] (-1+0.18, 0+0.18)--( 0-0.18, 1-0.18)[dotted];
\draw[shift={( 4, 4)}][violet][->] (-1+0.18, 0+0.18)--( 0-0.18, 1-0.18);
\draw[shift={( 5, 5)}][violet][->] (-1+0.18, 0+0.18)--( 0-0.18, 1-0.18);
\draw[shift={(-1+2,-1)}][violet][->] (-1+0.18, 0+0.18)--( 0-0.18, 1-0.18)[dotted];
\draw[shift={( 0+2, 0)}][violet][->] (-1+0.18, 0+0.18)--( 0-0.18, 1-0.18);
\draw[shift={( 1+2, 1)}][violet][->] (-1+0.18, 0+0.18)--( 0-0.18, 1-0.18);
\draw[shift={( 2+2, 2)}][violet][->] (-1+0.18, 0+0.18)--( 0-0.18, 1-0.18)[dotted];
\draw[shift={( 3+2, 3)}][violet][->] (-1+0.18, 0+0.18)--( 0-0.18, 1-0.18);
\draw[shift={( 4+2, 4)}][violet][->] (-1+0.18, 0+0.18)--( 0-0.18, 1-0.18);
\draw[shift={(-1+4,-1)}][violet][->] (-1+0.18, 0+0.18)--( 0-0.18, 1-0.18);
\draw[shift={( 0+4, 0)}][violet][->] (-1+0.18, 0+0.18)--( 0-0.18, 1-0.18);
\draw[shift={( 1+4, 1)}][violet][->] (-1+0.18, 0+0.18)--( 0-0.18, 1-0.18)[dotted];
\draw[shift={( 2+4, 2)}][violet][->] (-1+0.18, 0+0.18)--( 0-0.18, 1-0.18);
\draw[shift={( 3+4, 3)}][violet][->] (-1+0.18, 0+0.18)--( 0-0.18, 1-0.18);
\draw[shift={(-1+6,-1)}][violet][->] (-1+0.18, 0+0.18)--( 0-0.18, 1-0.18);
\draw[shift={( 0+6, 0)}][violet][->] (-1+0.18, 0+0.18)--( 0-0.18, 1-0.18)[dotted];
\draw[shift={( 1+6, 1)}][violet][->] (-1+0.18, 0+0.18)--( 0-0.18, 1-0.18);
\draw[shift={( 2+6, 2)}][violet][->] (-1+0.18, 0+0.18)--( 0-0.18, 1-0.18);
\draw[shift={(-1+8,-1)}][violet][->] (-1+0.18, 0+0.18)--( 0-0.18, 1-0.18)[dotted];
\draw[shift={( 0+8, 0)}][violet][->] (-1+0.18, 0+0.18)--( 0-0.18, 1-0.18);
\draw[shift={( 1+8, 1)}][violet][->] (-1+0.18, 0+0.18)--( 0-0.18, 1-0.18);
\draw[shift={(-1+10,-1)}][violet][->] (-1+0.18, 0+0.18)--( 0-0.18, 1-0.18);
\draw[shift={( 0+10, 0)}][violet][->] (-1+0.18, 0+0.18)--( 0-0.18, 1-0.18);
\draw[shift={(-1+12,-1)}][violet][->] (-1+0.18, 0+0.18)--( 0-0.18, 1-0.18);
\draw[shift={(12-0, 0+0)}][red][opacity=0.5][line width=0.8pt][->] (-1+0.18, 0-0.18)--( 0-0.18,-1+0.18);
\draw[shift={(12-1, 0+1)}][red][opacity=0.5][line width=0.8pt][->] (-1+0.18, 0-0.18)--( 0-0.18,-1+0.18);
\draw[shift={(12-2, 0+2)}][red][opacity=0.5][line width=0.8pt][->] (-1+0.18, 0-0.18)--( 0-0.18,-1+0.18)[dotted];
\draw[shift={(12-3, 0+3)}][red][opacity=0.5][line width=0.8pt][->] (-1+0.18, 0-0.18)--( 0-0.18,-1+0.18);
\draw[shift={(12-4, 0+4)}][red][opacity=0.5][line width=0.8pt][->] (-1+0.18, 0-0.18)--( 0-0.18,-1+0.18);
\draw[shift={(12-5, 0+5)}][red][opacity=0.5][line width=0.8pt][->] (-1+0.18, 0-0.18)--( 0-0.18,-1+0.18)[dotted];
\draw[shift={(12-6, 0+6)}][red][opacity=0.5][line width=0.8pt][->] (-1+0.18, 0-0.18)--( 0-0.18,-1+0.18);
\draw[shift={(10-0, 0+0)}][red][opacity=0.5][line width=0.8pt][->] (-1+0.18, 0-0.18)--( 0-0.18,-1+0.18);
\draw[shift={(10-1, 0+1)}][red][opacity=0.5][line width=0.8pt][->] (-1+0.18, 0-0.18)--( 0-0.18,-1+0.18)[dotted];
\draw[shift={(10-2, 0+2)}][red][opacity=0.5][line width=0.8pt][->] (-1+0.18, 0-0.18)--( 0-0.18,-1+0.18);
\draw[shift={(10-3, 0+3)}][red][opacity=0.5][line width=0.8pt][->] (-1+0.18, 0-0.18)--( 0-0.18,-1+0.18);
\draw[shift={(10-4, 0+4)}][red][opacity=0.5][line width=0.8pt][->] (-1+0.18, 0-0.18)--( 0-0.18,-1+0.18)[dotted];
\draw[shift={(10-5, 0+5)}][red][opacity=0.5][line width=0.8pt][->] (-1+0.18, 0-0.18)--( 0-0.18,-1+0.18);
\draw[shift={(8-0, 0+0)}][red][opacity=0.5][line width=0.8pt][->] (-1+0.18, 0-0.18)--( 0-0.18,-1+0.18)[dotted];
\draw[shift={(8-1, 0+1)}][red][opacity=0.5][line width=0.8pt][->] (-1+0.18, 0-0.18)--( 0-0.18,-1+0.18);
\draw[shift={(8-2, 0+2)}][red][opacity=0.5][line width=0.8pt][->] (-1+0.18, 0-0.18)--( 0-0.18,-1+0.18);
\draw[shift={(8-3, 0+3)}][red][opacity=0.5][line width=0.8pt][->] (-1+0.18, 0-0.18)--( 0-0.18,-1+0.18)[dotted];
\draw[shift={(8-4, 0+4)}][red][opacity=0.5][line width=0.8pt][->] (-1+0.18, 0-0.18)--( 0-0.18,-1+0.18);
\draw[shift={(6-0, 0+0)}][red][opacity=0.5][line width=0.8pt][->] (-1+0.18, 0-0.18)--( 0-0.18,-1+0.18);
\draw[shift={(6-1, 0+1)}][red][opacity=0.5][line width=0.8pt][->] (-1+0.18, 0-0.18)--( 0-0.18,-1+0.18);
\draw[shift={(6-2, 0+2)}][red][opacity=0.5][line width=0.8pt][->] (-1+0.18, 0-0.18)--( 0-0.18,-1+0.18)[dotted];
\draw[shift={(6-3, 0+3)}][red][opacity=0.5][line width=0.8pt][->] (-1+0.18, 0-0.18)--( 0-0.18,-1+0.18);
\draw[shift={(4-0, 0+0)}][red][opacity=0.5][line width=0.8pt][->] (-1+0.18, 0-0.18)--( 0-0.18,-1+0.18);
\draw[shift={(4-1, 0+1)}][red][opacity=0.5][line width=0.8pt][->] (-1+0.18, 0-0.18)--( 0-0.18,-1+0.18)[dotted];
\draw[shift={(4-2, 0+2)}][red][opacity=0.5][line width=0.8pt][->] (-1+0.18, 0-0.18)--( 0-0.18,-1+0.18);
\draw[shift={(2-0, 0+0)}][red][opacity=0.5][line width=0.8pt][->] (-1+0.18, 0-0.18)--( 0-0.18,-1+0.18)[dotted];
\draw[shift={(2-1, 0+1)}][red][opacity=0.5][line width=0.8pt][->] (-1+0.18, 0-0.18)--( 0-0.18,-1+0.18);
\draw[shift={(0-0, 0+0)}][red][opacity=0.5][line width=0.8pt][->] (-1+0.18, 0-0.18)--( 0-0.18,-1+0.18);
\end{tikzpicture}
  \caption{The Auslander-Reiten quiver of $A_n$}
  \label{fig:linaer-An}
\end{figure}

For any $1\leqslant m < n$, let
\[T_m = \Bigg(\bigoplus_{1\leqslant j\leqslant m}M_{[1,j]}\Bigg)
        \bigoplus  \Bigg(\bigoplus_{i\in(\dA_n)_0}M_{[i,i]}\Bigg). \]
The computing is divided to two steps as follows:

Step 1: describe all irreducible $T_m$-coghosts (see Proposition \ref{prop:Tm-cog});

Step 2: show that $\sup\ospec(A_n)=n-1$.

Thus, we can obtain the Orlov spectrum of $A_n$ is $\ospec(A_n)=\{1,2,\ldots,n-1\}$ by using Proposition \ref{prop-An}.

\begin{lemma} \label{lemm:Tm-cog1}
All $f^+_{[i,n]}: M_{[i,n]}\to M_{[i-1,n]}$ {\rm(}$m+1\leqslant i\leqslant n${\rm)} are $T_m$-coghosts.
\end{lemma}

\begin{proof}
First of all, for any $M_{[j,j]}$, we have $M_{[j,j]}$ is isomorphic to the simple module $S(j)$ corresponded by $j\in(\dA_n)_0$,
and so the following statements hold.
\begin{itemize}
  \item[(1)] If $j\leqslant i-1$ then $\Hom_{A_n}(M_{[i,n]}, M_{[j,j]})=0$.
  \item[(2)] If $j\geqslant i$, then $\Hom_{A_n}(M_{[i-1,n]}, M_{[j,j]})=0$.
\end{itemize}
The proof of (1) is trivial. Next we prove (2).
For any $1\leqslant p < q\leqslant n$ and $p< r\leqslant q$,
the morphism $h=(h_v)_{v\in(\dA_n)_0} \in \rep(\dA_n)$ from
the representation $((M_{[p,q]})_v, (M_{[p,q]})_{\alpha})_{v\in(\dA_n)_0, \alpha\in(\dA_n)_1}$ of $M_{[p,q]}$
to the simple representation of $S(r)$ is a family of
$\kk$-linear maps such that the following diagram commutes.
\[\xymatrix{
\cdots \ar[r] & (M_{[p,q]})_{r-1} \ar[r]^{d_{r-1}} \ar[d]
              & (M_{[p,q]})_r     \ar[r]^{d_r} \ar[d]_{h_r}
              & (M_{[p,q]})_{r+1} \ar[r]^{d_{r+1}} \ar[d] & \cdots \\
\cdots \ar[r] & 0 \ar[r] & S(j)e_r \ar[r] & 0 \ar[r] & \cdots
}\]
($e_r$ is the path of length zero corresponded by the vertex $r$ of $(\dA_n)_0$).
Then $r-1\geqslant p$, and so $(M_{[p,q]})_{r-1}=\kk$. Since $p\ne q$, we obtain $(M_{[p,q]})_r=\kk$, and so $d_{r-1} = 1_{\kk}$.
By the commutativity of the above diagram, we obtain $h_r=0$.
Take $p=i-1$, $q=n$ and $r=j\geqslant i$. Then any $h\in\Hom_{A_n}(M_{[i-1,n]}, M_{[j,j]})$ is zero holds for all $p=i-1< r=j \leqslant q=n$.
Thus, the statement (2) holds.
Then, by Lemma \ref{lemm:cog-criterion}, every $f^+_{[i,n]}$ ($m+1\leqslant i\leqslant n$) is an $M_{[j,j]}$-coghost ($j\in(\dA_n)_0$).
Furthermore, we obtain that it is also a $\big(\bigoplus_{j\in(\dA_n)_0}M_{[j,j]}\big)$-coghost
by Lemma \ref{lemm:cog-directsum}.

Next, we prove that every $f^+_{[i,n]}$ ($m+1\leqslant i\leqslant n$) is a
$\big(\bigoplus_{1\leqslant j\leqslant n}M_{[1,j]}\big)$-coghost ($1\leqslant j\leqslant m$).
For any $g\in\Hom_{A_n}(M_{[i-1,n]}, M_{[1,j]})$, note that $i>j$, then it is obvious that $g$ is zero.
In this case, $gf^+_{[i,n]}=0$, then $f^+_{[i,n]}$ is an $M_{[1,j]}$-coghost by Lemma \ref{lemm:cog-criterion}. Thus, $f^+_{[i,n]}$ is also a
$\big(\bigoplus_{1\leqslant j\leqslant n}M_{[1,j]}\big)$-coghost
by Lemma \ref{lemm:cog-directsum}.

Finally, using Lemma \ref{lemm:cog-directsum} again, we have $f^+_{[i,n]}$ is a $T_m$-coghost.
\end{proof}

\begin{lemma}\label{lemm:Tm-cog2}
If $f^+_{[r,s]}$ $(s\geqslant 1)$ is a $T_m$-coghost, then so is $f^+_{[r,s-1]}$.
\end{lemma}

\begin{proof}
For any morphism $h: M_{[r-1,s-1]}\to T'$ with $T'\in\add T_m$, we have
\[(h\circ f^-_{[r-1,s]}) \circ f^+_{[r,s]}=0\]
since $h\circ f^-_{[r-1,s]}$ is a morphism in $\Hom_{A_n}(M_{[r-1,s]},T')$ and $f^+_{[r,s]}$ is a $T_m$-coghost.
\begin{figure}[htbp]
\centering
\begin{tikzpicture}[xscale=1.8,yscale=1.2]
\draw[->] (-1+0.18, 0+0.18)--( 0-0.18, 1-0.18); \draw (-0.5, 0.5) node[above  left]{$f^+_{[r,s]}$};
\draw[->] ( 0+0.18,-1+0.18)--( 1-0.18, 0-0.18); \draw ( 0.5,-0.5) node[below right]{$f^+_{[r,s-1]}$};
\draw[->] ( 0+0.18, 1-0.18)--( 1-0.18, 0+0.18); \draw ( 0.5, 0.5) node[above right]{$f^-_{[r-1,s]}$};
\draw[->] (-1+0.18, 0-0.18)--( 0-0.18,-1+0.18); \draw (-0.5,-0.5) node[below  left]{$f^-_{[r,s]}$};
\draw (-1,0) node{$M_{[r,s]}$} (0,1) node{$M_{[r-1,s]}$} (0,-1) node{$M_{[r,s-1]}$} (1,0) node{$M_{[r-1,s-1]}$};
\end{tikzpicture}
\caption{Anti-commutative block in $\Gamma(\mod A_n)$}
\label{fig:Tm-cog2}
\end{figure}
Then $0 = h\circ f^-_{[r-1,s]}\circ f^+_{[r,s]} = -h\circ f^+_{[r,s-1]}\circ f^-_{[r,s]}$ (see \Pic\ \ref{fig:Tm-cog2}).
Since $f^-_{[r,s]}$ is an epimorphism, we obtain $hf^+_{[r,s-1]}=0$.
\end{proof}

The following result is given by Lemmas \ref{lemm:Tm-cog1} and \ref{lemm:Tm-cog2}, which provide all irreducible $T_m$-coghosts.

\begin{proposition} \label{prop:Tm-cog}
The morphisms $f^+_{[i,n-t]}$ {\rm(}$m+1\leqslant i\leqslant n$, $0\leqslant t<n${\rm)} constitute all irreducible $T_m$-coghosts in $\mod A_n$.
\end{proposition}

\begin{proof}
Let $\mathfrak{F}=\{f^+_{[i,n-t]} \mid m+1\leqslant i\leqslant n, 0\leqslant t<n\}$.
We show that a homomorphism $f$ is an irreducible $T_m$-coghosts if and only if $f\in \mathfrak{F}$.

First of all, for any $1\leqslant t< j$ $(\leqslant m)$, $f^+_{[t+1,j]}: M_{[t+1,j]}\to M_{[t,j]}$
(see the irreducible morphisms $f^+_{[-,-]}$ in the {\color{green}part A} of \Pic\ \ref{fig:Tm-cog})
is not a $T_m$-coghost. Indeed, every $f^+_{[i,j]}$ ($2\leqslant i\leqslant j$) is monomorphic,
then $h = \prod\limits_{i=2}^t f^+_{[i,j]}: M_{[t,j]} \to M_{[1,j]}$ and $hf^+_{[t+1,j]}$ are monomorphic,
and then $hf^+_{[t+1,j]}\ne 0$. Since $M_{[1,j]}\in\add T$, $f^+_{[t+1,j]}$ is not a $T_m$-coghost.

Secondly, for any $f^+_{[r,s]}$, if $f^+_{[r,s]}$ is not a $T_m$-coghost, then by Lemma \ref{lemm:Tm-cog2},
we obtain that $f^+_{[r,s+1]}$ is not a $T_m$-coghost.
Thus any irreducible morphisms $f^+_{[-,-]}$ in the {\color{blue}part B} of \Pic\ \ref{fig:Tm-cog}.
Therefore, $f^+_{[r,s]}$ is an irreducible $T_m$-coghosts if and only if $f^+_{[r,s]}\in \mathfrak{F}$.

Finally, any $f^-_{[r,s]}:M_{[r,s]}\to M_{[r,s-1]}$ is not a $T_m$-coghost because $\mathfrak{e}_{[r,s-1]}f^-_{[r,s]}\ne 0$,
where $\mathfrak{e}_{[r,s-1]}: M_{[r,s-1]} \to \top(M_{[r,s-1]})$ is the natural epimorphism
from $M_{[r,s-1]}$ to its top $\top(M_{[r,s-1]})\cong M_{[r,r]} \in \add T_m$.
Thus, $f^-_{[r,s]}\notin\mathfrak{F}$.

In conclusion, $f$ is an irreducible $T_m$-coghosts if and only if $f\in \mathfrak{F}$.
\end{proof}

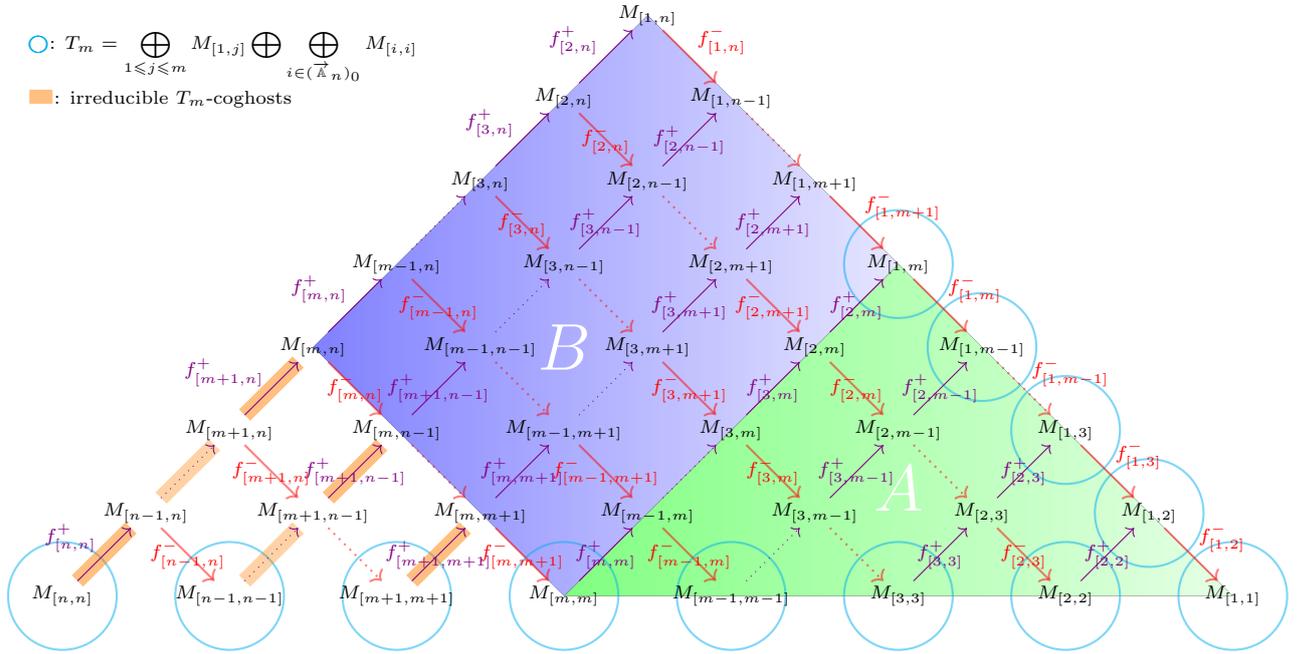
\begin{figure}[htbp]
\centering
\hspace{-1cm}
\begin{tikzpicture}[scale=1.1]\tiny
\fill[left color=green!50, right color=green!10] (8, 3) -- ( 4,-1) -- (12,-1) -- ( 8, 3);
\draw[white] (8,0.3) node[font=\fontsize{30}{0}\selectfont]{$A$};
\fill[left color=blue!50, right color=blue!10] (8, 3) -- (5,6) -- (1,2) -- ( 4,-1);
\draw[white] (4,2) node[font=\fontsize{30}{0}\selectfont]{$B$};
\draw[cyan][line width=0.8pt][opacity=0.5]
     (-2,-1) circle(0.65) ( 0,-1) circle(0.65) ( 2,-1) circle(0.65) ( 4,-1) circle(0.65)
     ( 6,-1) circle(0.65) ( 8,-1) circle(0.65) (10,-1) circle(0.65) (12,-1) circle(0.65)
     (11, 0) circle(0.65) (10, 1) circle(0.65) ( 9, 2) circle(0.65) ( 8, 3) circle(0.65);
\draw[orange][shift={(-1,-1)}][line width=6pt][opacity=0.50] (-1+0.18, 0+0.18)--( 0-0.18, 1-0.18);
\draw[orange][shift={( 0, 0)}][line width=6pt][opacity=0.35] (-1+0.18, 0+0.18)--( 0-0.18, 1-0.18);
\draw[orange][shift={( 1, 1)}][line width=6pt][opacity=0.50] (-1+0.18, 0+0.18)--( 0-0.18, 1-0.18);
\draw[orange][shift={( 1,-1)}][line width=6pt][opacity=0.35] (-1+0.18, 0+0.18)--( 0-0.18, 1-0.18);
\draw[orange][shift={( 2, 0)}][line width=6pt][opacity=0.50] (-1+0.18, 0+0.18)--( 0-0.18, 1-0.18);
\draw[orange][shift={( 3,-1)}][line width=6pt][opacity=0.50] (-1+0.18, 0+0.18)--( 0-0.18, 1-0.18);
\draw (-2+0,-1+0) node{$M_{[n  , n  ]}$}
      (-2+1,-1+1) node{$M_{[n-1, n  ]}$}
      (-2+2,-1+2) node{$M_{[m+1, n  ]}$}
      (-2+3,-1+3) node{$M_{[m  , n  ]}$}
      (-2+4,-1+4) node{$M_{[m-1, n  ]}$}
      (-2+5,-1+5) node{$M_{[3  , n  ]}$}
      (-2+6,-1+6) node{$M_{[2  , n  ]}$}
      (-2+7,-1+7) node{$M_{[1  , n  ]}$};
\draw ( 0+0,-1+0) node{$M_{[n-1, n-1]}$}
      ( 0+1,-1+1) node{$M_{[m+1, n-1]}$}
      ( 0+2,-1+2) node{$M_{[m  , n-1]}$}
      ( 0+3,-1+3) node{$M_{[m-1, n-1]}$}
      ( 0+4,-1+4) node{$M_{[3  , n-1]}$}
      ( 0+5,-1+5) node{$M_{[2  , n-1]}$}
      ( 0+6,-1+6) node{$M_{[1  , n-1]}$};
\draw ( 2+0,-1+0) node{$M_{[m+1, m+1]}$}
      ( 2+1,-1+1) node{$M_{[m  , m+1]}$}
      ( 2+2,-1+2) node{$M_{[m-1, m+1]}$}
      ( 2+3,-1+3) node{$M_{[3  , m+1]}$}
      ( 2+4,-1+4) node{$M_{[2  , m+1]}$}
      ( 2+5,-1+5) node{$M_{[1  , m+1]}$};
\draw ( 4+0,-1+0) node{$M_{[m  , m  ]}$}
      ( 4+1,-1+1) node{$M_{[m-1, m  ]}$}
      ( 4+2,-1+2) node{$M_{[3  , m  ]}$}
      ( 4+3,-1+3) node{$M_{[2  , m  ]}$}
      ( 4+4,-1+4) node{$M_{[1  , m  ]}$};
\draw ( 6+0,-1+0) node{$M_{[m-1, m-1]}$}
      ( 6+1,-1+1) node{$M_{[3  , m-1]}$}
      ( 6+2,-1+2) node{$M_{[2  , m-1]}$}
      ( 6+3,-1+3) node{$M_{[1  , m-1]}$};
\draw ( 8+0,-1+0) node{$M_{[3  , 3  ]}$}
      ( 8+1,-1+1) node{$M_{[2  , 3  ]}$}
      ( 8+2,-1+2) node{$M_{[1  , 3  ]}$};
\draw (10+0,-1+0) node{$M_{[2  , 2  ]}$}
      (10+1,-1+1) node{$M_{[1  , 2  ]}$};
\draw (12+0,-1+0) node{$M_{[1  , 1  ]}$};
\draw[violet] (-1.5+0,-0.3+0) node[left]{$f^+_{[n  , n  ]}$}
              (-1.5+2,-0.3+2) node[left]{$f^+_{[m+1, n  ]}$}
              (-1.5+3,-0.3+3) node[left]{$f^+_{[m  , n  ]}$}
              (-1.5+5,-0.3+5) node[left]{$f^+_{[3  , n  ]}$}
              (-1.5+6,-0.3+6) node[left]{$f^+_{[2  , n  ]}$};
\draw[violet] ( 0.5+1,-0.5+1) node{$f^+_{[m+1, n-1]}$}
              ( 0.5+2,-0.5+2) node{$f^+_{[m+1, n-1]}$}
              ( 0.5+4,-0.5+4) node{$f^+_{[3  , n-1]}$}
              ( 0.5+5,-0.5+5) node{$f^+_{[2  , n-1]}$};
\draw[violet] ( 2.5+0,-0.5+0) node{$f^+_{[m+1, m+1]}$}
              ( 2.5+1,-0.5+1) node{$f^+_{[m  , m+1]}$}
              ( 2.5+3,-0.5+3) node{$f^+_{[3  , m+1]}$}
              ( 2.5+4,-0.5+4) node{$f^+_{[2  , m+1]}$};
\draw[violet] ( 4.5+0,-0.5+0) node{$f^+_{[m  , m  ]}$}
              ( 4.5+2,-0.5+2) node{$f^+_{[3  , m  ]}$}
              ( 4.5+3,-0.5+3) node{$f^+_{[2  , m  ]}$};
\draw[violet] ( 6.5+1,-0.5+1) node{$f^+_{[3  , m-1]}$}
              ( 6.5+2,-0.5+2) node{$f^+_{[2  , m-1]}$};
\draw[violet] ( 8.5+0,-0.5+0) node{$f^+_{[3  , 3  ]}$}
              ( 8.5+1,-0.5+1) node{$f^+_{[2  , 3  ]}$};
\draw[violet] (10.5+0,-0.5+0) node{$f^+_{[2  , 2  ]}$};
\draw[   red] (11.5-0,-0.3+0) node[right]{$f^-_{[1  , 2  ]}$}
              (11.5-1,-0.3+1) node[right]{$f^-_{[1  , 3  ]}$}
              (11.5-2,-0.3+2) node[right]{$f^-_{[1  , m-1]}$}
              (11.5-3,-0.3+3) node[right]{$f^-_{[1  , m  ]}$}
              (11.5-4,-0.3+4) node[right]{$f^-_{[1  , m+1]}$}
              (11.5-6,-0.3+6) node[right]{$f^-_{[1  , n  ]}$};
\draw[   red] ( 9.5-0,-0.5+0) node{$f^-_{[2  , 3  ]}$}
              ( 9.5-2,-0.5+2) node{$f^-_{[2  , m  ]}$}
              ( 9.5-3,-0.5+3) node{$f^-_{[2  , m+1]}$}
              ( 9.5-5,-0.5+5) node{$f^-_{[2  , n  ]}$};
\draw[   red] ( 7.5-1,-0.5+1) node{$f^-_{[3  , m  ]}$}
              ( 7.5-2,-0.5+2) node{$f^-_{[3  , m+1]}$}
              ( 7.5-4,-0.5+4) node{$f^-_{[3  , n  ]}$};
\draw[   red] ( 5.5-0,-0.5+0) node{$f^-_{[m-1, m  ]}$}
              ( 5.5-1,-0.5+1) node{$f^-_{[m-1, m+1]}$}
              ( 5.5-3,-0.5+3) node{$f^-_{[m-1, n  ]}$};
\draw[   red] ( 3.5-0,-0.5+0) node{$f^-_{[m  , m+1]}$}
              ( 3.5-2,-0.5+2) node{$f^-_{[m  , n  ]}$};
\draw[   red] ( 1.5-1,-0.5+1) node{$f^-_{[m+1, n  ]}$};
\draw[   red] (-0.5-0,-0.5+0) node{$f^-_{[n-1, n  ]}$};
%
\draw[shift={(-1,-1)}][violet][->] (-1+0.18, 0+0.18)--( 0-0.18, 1-0.18);
\draw[shift={( 0, 0)}][violet][->] (-1+0.18, 0+0.18)--( 0-0.18, 1-0.18)[dotted];
\draw[shift={( 1, 1)}][violet][->] (-1+0.18, 0+0.18)--( 0-0.18, 1-0.18);
\draw[shift={( 2, 2)}][violet][->] (-1+0.18, 0+0.18)--( 0-0.18, 1-0.18);
\draw[shift={( 3, 3)}][violet][->] (-1+0.18, 0+0.18)--( 0-0.18, 1-0.18)[dotted];
\draw[shift={( 4, 4)}][violet][->] (-1+0.18, 0+0.18)--( 0-0.18, 1-0.18);
\draw[shift={( 5, 5)}][violet][->] (-1+0.18, 0+0.18)--( 0-0.18, 1-0.18);
\draw[shift={(-1+2,-1)}][violet][->] (-1+0.18, 0+0.18)--( 0-0.18, 1-0.18)[dotted];
\draw[shift={( 0+2, 0)}][violet][->] (-1+0.18, 0+0.18)--( 0-0.18, 1-0.18);
\draw[shift={( 1+2, 1)}][violet][->] (-1+0.18, 0+0.18)--( 0-0.18, 1-0.18);
\draw[shift={( 2+2, 2)}][violet][->] (-1+0.18, 0+0.18)--( 0-0.18, 1-0.18)[dotted];
\draw[shift={( 3+2, 3)}][violet][->] (-1+0.18, 0+0.18)--( 0-0.18, 1-0.18);
\draw[shift={( 4+2, 4)}][violet][->] (-1+0.18, 0+0.18)--( 0-0.18, 1-0.18);
\draw[shift={(-1+4,-1)}][violet][->] (-1+0.18, 0+0.18)--( 0-0.18, 1-0.18);
\draw[shift={( 0+4, 0)}][violet][->] (-1+0.18, 0+0.18)--( 0-0.18, 1-0.18);
\draw[shift={( 1+4, 1)}][violet][->] (-1+0.18, 0+0.18)--( 0-0.18, 1-0.18)[dotted];
\draw[shift={( 2+4, 2)}][violet][->] (-1+0.18, 0+0.18)--( 0-0.18, 1-0.18);
\draw[shift={( 3+4, 3)}][violet][->] (-1+0.18, 0+0.18)--( 0-0.18, 1-0.18);
\draw[shift={(-1+6,-1)}][violet][->] (-1+0.18, 0+0.18)--( 0-0.18, 1-0.18);
\draw[shift={( 0+6, 0)}][violet][->] (-1+0.18, 0+0.18)--( 0-0.18, 1-0.18)[dotted];
\draw[shift={( 1+6, 1)}][violet][->] (-1+0.18, 0+0.18)--( 0-0.18, 1-0.18);
\draw[shift={( 2+6, 2)}][violet][->] (-1+0.18, 0+0.18)--( 0-0.18, 1-0.18);
\draw[shift={(-1+8,-1)}][violet][->] (-1+0.18, 0+0.18)--( 0-0.18, 1-0.18)[dotted];
\draw[shift={( 0+8, 0)}][violet][->] (-1+0.18, 0+0.18)--( 0-0.18, 1-0.18);
\draw[shift={( 1+8, 1)}][violet][->] (-1+0.18, 0+0.18)--( 0-0.18, 1-0.18);
\draw[shift={(-1+10,-1)}][violet][->] (-1+0.18, 0+0.18)--( 0-0.18, 1-0.18);
\draw[shift={( 0+10, 0)}][violet][->] (-1+0.18, 0+0.18)--( 0-0.18, 1-0.18);
\draw[shift={(-1+12,-1)}][violet][->] (-1+0.18, 0+0.18)--( 0-0.18, 1-0.18);
\draw[shift={(12-0, 0+0)}][red][opacity=0.5][line width=0.8pt][->] (-1+0.18, 0-0.18)--( 0-0.18,-1+0.18);
\draw[shift={(12-1, 0+1)}][red][opacity=0.5][line width=0.8pt][->] (-1+0.18, 0-0.18)--( 0-0.18,-1+0.18);
\draw[shift={(12-2, 0+2)}][red][opacity=0.5][line width=0.8pt][->] (-1+0.18, 0-0.18)--( 0-0.18,-1+0.18)[dotted];
\draw[shift={(12-3, 0+3)}][red][opacity=0.5][line width=0.8pt][->] (-1+0.18, 0-0.18)--( 0-0.18,-1+0.18);
\draw[shift={(12-4, 0+4)}][red][opacity=0.5][line width=0.8pt][->] (-1+0.18, 0-0.18)--( 0-0.18,-1+0.18);
\draw[shift={(12-5, 0+5)}][red][opacity=0.5][line width=0.8pt][->] (-1+0.18, 0-0.18)--( 0-0.18,-1+0.18)[dotted];
\draw[shift={(12-6, 0+6)}][red][opacity=0.5][line width=0.8pt][->] (-1+0.18, 0-0.18)--( 0-0.18,-1+0.18);
\draw[shift={(10-0, 0+0)}][red][opacity=0.5][line width=0.8pt][->] (-1+0.18, 0-0.18)--( 0-0.18,-1+0.18);
\draw[shift={(10-1, 0+1)}][red][opacity=0.5][line width=0.8pt][->] (-1+0.18, 0-0.18)--( 0-0.18,-1+0.18)[dotted];
\draw[shift={(10-2, 0+2)}][red][opacity=0.5][line width=0.8pt][->] (-1+0.18, 0-0.18)--( 0-0.18,-1+0.18);
\draw[shift={(10-3, 0+3)}][red][opacity=0.5][line width=0.8pt][->] (-1+0.18, 0-0.18)--( 0-0.18,-1+0.18);
\draw[shift={(10-4, 0+4)}][red][opacity=0.5][line width=0.8pt][->] (-1+0.18, 0-0.18)--( 0-0.18,-1+0.18)[dotted];
\draw[shift={(10-5, 0+5)}][red][opacity=0.5][line width=0.8pt][->] (-1+0.18, 0-0.18)--( 0-0.18,-1+0.18);
\draw[shift={(8-0, 0+0)}][red][opacity=0.5][line width=0.8pt][->] (-1+0.18, 0-0.18)--( 0-0.18,-1+0.18)[dotted];
\draw[shift={(8-1, 0+1)}][red][opacity=0.5][line width=0.8pt][->] (-1+0.18, 0-0.18)--( 0-0.18,-1+0.18);
\draw[shift={(8-2, 0+2)}][red][opacity=0.5][line width=0.8pt][->] (-1+0.18, 0-0.18)--( 0-0.18,-1+0.18);
\draw[shift={(8-3, 0+3)}][red][opacity=0.5][line width=0.8pt][->] (-1+0.18, 0-0.18)--( 0-0.18,-1+0.18)[dotted];
\draw[shift={(8-4, 0+4)}][red][opacity=0.5][line width=0.8pt][->] (-1+0.18, 0-0.18)--( 0-0.18,-1+0.18);
\draw[shift={(6-0, 0+0)}][red][opacity=0.5][line width=0.8pt][->] (-1+0.18, 0-0.18)--( 0-0.18,-1+0.18);
\draw[shift={(6-1, 0+1)}][red][opacity=0.5][line width=0.8pt][->] (-1+0.18, 0-0.18)--( 0-0.18,-1+0.18);
\draw[shift={(6-2, 0+2)}][red][opacity=0.5][line width=0.8pt][->] (-1+0.18, 0-0.18)--( 0-0.18,-1+0.18)[dotted];
\draw[shift={(6-3, 0+3)}][red][opacity=0.5][line width=0.8pt][->] (-1+0.18, 0-0.18)--( 0-0.18,-1+0.18);
\draw[shift={(4-0, 0+0)}][red][opacity=0.5][line width=0.8pt][->] (-1+0.18, 0-0.18)--( 0-0.18,-1+0.18);
\draw[shift={(4-1, 0+1)}][red][opacity=0.5][line width=0.8pt][->] (-1+0.18, 0-0.18)--( 0-0.18,-1+0.18)[dotted];
\draw[shift={(4-2, 0+2)}][red][opacity=0.5][line width=0.8pt][->] (-1+0.18, 0-0.18)--( 0-0.18,-1+0.18);
\draw[shift={(2-0, 0+0)}][red][opacity=0.5][line width=0.8pt][->] (-1+0.18, 0-0.18)--( 0-0.18,-1+0.18)[dotted];
\draw[shift={(2-1, 0+1)}][red][opacity=0.5][line width=0.8pt][->] (-1+0.18, 0-0.18)--( 0-0.18,-1+0.18);
\draw[shift={(0-0, 0+0)}][red][opacity=0.5][line width=0.8pt][->] (-1+0.18, 0-0.18)--( 0-0.18,-1+0.18);
\draw(-2.5, 5.5) node[right]{${\color{cyan}\pmb{\bigcirc}}$:
                $T_m = \displaystyle \bigoplus_{1\leqslant j\leqslant m}M_{[1,j]}\bigoplus\bigoplus_{i\in(\dA_n)_0}M_{[i,i]}$};
\draw(-2.5, 5.0) node[right]{${\color{orange!50}\blacksquare\!\!\blacksquare}$: irreducible $T_m$-coghosts};
\end{tikzpicture}
\caption{The Auslander-Reiten quiver of $A_n$ ($T_m$-coghosts)}
\label{fig:Tm-cog}
\end{figure}

\begin{lemma} \label{lemm:An-compo=0}
For any $n$ homomorphisms $f_i: X_i\to X_{i+1}$ $(1\leqslant i\leqslant n)$ in $\rad_{A_n}(X_i,X_{i+1})$,
we have $f_n\circ \cdots\circ f_2\circ f_1=0$.
\end{lemma}

\begin{proof}
Consider the Auslander algebra $A_n^{\Aus}=\kk\Gamma(\mod A_n)/\I(\mod A_n)$ of $A_n$,
where $\I(\mod A_n)$ is generated by
\begin{itemize}
  \item[(1)]
    the zero relations $f^-_{[t-1,t]}f^+_{[t,t]}$ ($1< t\leqslant n$)
  \item[(2)]
    and the commutative relations
    $f^-_{[r-1,s]}f^+_{[r,s]}+f^-_{[r,s]}f^+_{[r,s-1]}$
    ($2\leqslant r\leqslant n-1$, $3\leqslant s\leqslant n$).
\end{itemize}
Any morphism $f_i$ is an element in $A_n^{\Aus}$, then there are $k_{i1}$, $\cdots$, $k_{it_i} \in \kk$ such that
\[f_i = \sum_{u=1}^{t_i} k_{iu}\wp_{iu},\]
where every $\wp_{iu}$ is a path on the bound quiver $(\Gamma(\mod A_n), \I(\mod A_n))$ of $A_n^{\Aus}$
whose length $\ell(\wp_{iu})$ satisfies the following conditions:
\begin{itemize}
  \item $\ell(\wp_{iu})\geqslant 1$ because $f_i\in \rad_{A_n}(X_i,X_{i+1}) \subseteq \rad(A_{n}^{\Aus})$;
  \item and $\ell(\wp_{iu})\leqslant n-1$ because $f^+_{[2,n]}\circ\cdots \circ f^+_{[n,n]}$ and $f^-_{[1,2]}\circ\cdots\circ f^-_{[1,n]}$
    are the longest paths in $(\Gamma(\mod A_n), \I(\mod A_n))$ whose length are $n-1$.
\end{itemize}
Therefore, $f_n\cdots f_1\in \rad^n(A_n^{\Aus})=0$, to be more precise, we have
\[ f_n\circ\cdots\circ f_2 \circ f_1
= \sum_{(u_1, \cdots, u_n)\in (\mathbb{N}^+_{\leqslant n})^{\times n}}
  \prod_{i=1}^n k_{i u_i}\wp_{i u_i}. \]
Since $\ell(\wp_{n u_n}\cdots \wp_{1 u_1})$ equals to either $\sum_{i=1}^n\ell(\wp_{i u_i})$ $(\geqslant n)$ or zero,
we have $\wp_{n u_n}\cdots \wp_{1 u_1}=0$, and so $f_n\circ \cdots \circ f_1=0$.
\end{proof}

\begin{proposition} \label{prop:max-ospecAn}
$\sup(\ospec(A_n)) = n-1$.
\end{proposition}

\begin{proof}
We have $n-1\in\ospec(A_n)$ by Proposition \ref{prop-An}.
Next, assume that there is an integer $y\geqslant n$ such that $y\in \ospec(A_n)$,
then there is a strong generator $T$ satisfying $[T]_y=\mod(A_n)$,
and $[T]_y\backslash [T]_{n-1}$ contains at least one non-zero $A_n$-module $Y$
such that there exists a homomorphism
$g\in \Gh_{T}^{y}(X, Y)$ which is not zero,
where $X$ is a non-zero $A_n$-module (see Coghost Lemma \ref{coghostlemm}).
Therefore, we can find $y$ $T$-coghosts
\[f_1: X=X_1\to X_2,\ f_2:X_2\to X_3,\ \ldots,\ f_y: X_y\to X_{y+1}=Y\]
satisfying the following two conditions:
\begin{itemize}
  \item all $f_i$ are not isomorphisms (by Lemma \ref{f_is_not_coghost});
  \item $g=f_y\circ\cdots \circ f_2\circ f_1 \ne 0$.
\end{itemize}
Consider the  Auslander algebra
\[A_{n}^{\Aus} = \End_{A_n}\bigg(\bigoplus_{M\in \mathrm{ind}(\mod(A_n))} M \bigg) \]
of $A_n$ whose quiver $\Gamma(A_n) = (\Gamma(A_n)_0, \Gamma(A_n)_1, \s, \t)$
is shown in Figure \ref{fig:linaer-An}.
Then each homomorphism $f_i$ can be seen as an element in $A_n^{\Aus}$.
In the above sense, $f_i$ is either a homomorphism lying in $\rad(A_n^{\Aus})$
or a homomorphism of the form
\[ \big(f_i = \big(f_{i1}\ f_{i2}\big): X_{i} = X_{i1}\oplus X_{i2} \To{} X_{i+1}\big)
   \in \kk\Gamma(A_n)_0 \oplus \rad(A^{\Aus}), \]
such that $f_{i1}\in\kk\Gamma(A_n)_0, \ f_{i2}\in \rad(A^{\Aus})$.
Then for any $h: X_{i+1} \to T'$, $T'\in\add T$, we have
\[h\circ f_i = \big( h\circ f_{i1}\ \ h\circ f_{i2} \big) = 0, \]
since $f_i$ is a $T$-coghost, and then $h\circ f_{i1}=0$ and $h\circ f_{12}=0$.
The above follows that  $f_{i1}$ is a $T$-coghost, it contradicts with Lemma \ref{f_is_not_coghost}.
$f_{i1}$ is an isomorphism defined on some semi-simple module.
Then each homomorphism $f_i$, as an element in $A^{\Aus}$, lies in $\rad(A^{\Aus})$,
that is, $f_i\in\rad_{A_n} (X_i, X_{i+1})$. It contradicts with Lemma \ref{lemm:An-compo=0}.
\end{proof}

\begin{theorem} \label{thm:main 3}
 $\ospec (A_{n})=\{0,1,2,\cdots,n-1\}.$
\end{theorem}

\begin{proof}
This theorem is obtained by Propositions \ref{prop-An} and \ref{prop:max-ospecAn} immediately.
\end{proof}

\section*{Acknowledgements}

The authors are greatly indebted to Pat Lank for helpful suggestions.

\section*{Funding}

This work is supported by National Natural Science Foundation of China (Grant Nos. 12401042, 12171207),
Guizhou Provincial Basic Research Program (Natural Science) (Grant No. ZK[2024]YiBan066),
and Scientific Research Foundation of Guizhou University (Grant Nos. [2022]53, [2022]65, [2023]16).







\end{document}